\documentclass{amsart}
\usepackage{amsmath, centernot}
\usepackage{graphicx}
\usepackage{booktabs}
\usepackage{setspace}
\usepackage{amsmath,amssymb}
\usepackage{amsfonts}
\usepackage{amsthm}
\usepackage{amscd}
\usepackage{mathtools}
\usepackage{braket}
\usepackage{tikz}
\usepackage{tikz-cd}

\usepackage{amssymb,amsthm}
\usepackage{color}
\usepackage{amscd}
\usepackage[all,cmtip]{xy}
\input xy
\xyoption{all}
\newtheorem{Lemma}{Lemma}[section]
\newtheorem{remark}[Lemma]{Remark}
\newtheorem{remarks}[Lemma]{Remarks}

\newtheorem{lemma}[Lemma]{Lemma}
\newtheorem{proposition}[Lemma]{Proposition}
\newtheorem{corollary}[Lemma]{Corollary}
\newtheorem{definition}[Lemma]{Definition}
\newtheorem{example}[Lemma]{Example}

\newcommand{\Cal}[1]{{\mathcal #1}}

\newcommand{\Hom}{\operatorname{hom}}
\DeclareMathOperator{\Triv}{\bf Triv}

\newcommand{\cal}{\mathcal}
\DeclareMathOperator{\op}{op}

\DeclareMathOperator{\Equiv}{\bf Equiv}
\DeclareMathOperator{\ParOrd}{\bf ParOrd}
\DeclareMathOperator{\pre}{\bf Preord}

\newcommand{\cmat}{\left(\begin{array}}
\newcommand{\fmat}{\end{array}\right)}

\newcommand{\Sets}{\mathsf{Set}}
\newcommand{\Top}{\mathsf{Top}}

\begin{document}
   \title{Pretorsion theories in general categories}
  \author[Alberto Facchini]{Alberto Facchini}
\address{Universit\`a di Padova, Dipartimento di Matematica ``Tullio Levi-Civita'', 35121 Padova, Italy}
 \email{facchini@math.unipd.it}
\thanks{ The first author was partially supported by Ministero dell'Istruzione, dell'Universit\`a e della Ricerca (Progetto di ricerca di rilevante interesse nazionale ``Categories, Algebras: Ring-Theoretical and Homological Approaches (CARTHA)''), Fondazione Cariverona (Research project ``Reducing complexity in algebra, logic, combinatorics - REDCOM'' within the framework of the programme “Ricerca Scientifica di Eccellenza 2018”), and the Mathematics Department ``Tullio Levi-Civita'' of the University of Padua (Research program DOR1828909 ``Anelli e categorie di moduli'').	
	The second author was partially supported by GNSAGA of Istituto Nazionale di Alta Matematica, Fondazione Cariverona (Research project ``Reducing complexity in algebra, logic, combinatorics - REDCOM'' within the framework of the programme “Ricerca Scientifica di Eccellenza 2018”), by the Mathematics Department ``Tullio Levi-Civita'' of the University of Padua (Research program DOR1828909 ``Anelli e categorie di moduli''), and  by the Department of Mathematics and Computer Science of the University of Catania (Research program “Propriet\`a
	algebriche locali e globali di anelli associati a curve e ipersuperfici” PTR 2016-18).
	The third author was supported by a ``Visiting scientist scholarship - year 2019'' by the University of Padua, and by a collaboration project \emph{Fonds d'Appui \`a l'Internationalisation} of the Universit\'e catholique de Louvain. }  

 \author[Carmelo Finocchiaro]{Carmelo Finocchiaro}
\address{Universit\`a di Catania, Dipartimento di Matematica e Informatica, 95125 Catania, Italy}
 \email{cafinocchiaro@unict.it}
 \author[Marino Gran]{Marino Gran}
\address{Universit\'{e} catholique de Louvain, Institut de Recherche en Math\'{e}matique et Physique, 1348 Louvain-la-Neuve, Belgium}
 \email{marino.gran@uclouvain.be}
      \begin{abstract}{ We present a setting for the study of torsion theories in general categories. The idea is to associate, with any pair ($\Cal T$, $\Cal F$) of full replete subcategories in a category $\Cal C$, the corresponding full subcategory $\Cal Z = \cal T \cap \cal F$ of \emph{trivial objects} in $\Cal C$. The morphisms which factor through $\Cal Z$ are called $\Cal Z$-trivial, and these form an ideal of morphisms, with respect to which one can define $\Cal Z$-prekernels, $\Cal Z$-precokernels, and short $\Cal Z$-preexact sequences. This naturally leads to the notion of pretorsion theory, which is the object of study of this article, and includes the classical one in the abelian context when $\Cal Z$ is reduced to the $0$-object of $\mathcal C$. We study the basic properties of pretorsion theories, and examine some new examples in the category of all endomappings of finite sets and in the category of preordered sets.}\end{abstract}

    \maketitle

\section{Introduction}

We present a setting for the study of torsion theories in general categories. Torsion theories in special non-additive categories have been studied by several authors in the past forty years (see for instance \cite{B, BG,BGV,BV,CDT,EG, GJ,GJM,JT,BR, RT,V82,V84,VW},  and the references therein). These special cases concern categories that are pointed, or homological, or semi-abelian, or with equalizers and pushouts, and so on. Our more general approach is based on the idea that a (pre)torsion theory {is simply a} pair $(\Cal T, \Cal F)$ of {full (replete) subcategories of a category $\Cal C$, with $\Cal T$ the \emph{torsion} subcategory and $\mathcal F$ the \emph{torsion-free} subcategory}, for which the ideal of morphisms (in the sense of \cite{Eh}) consists of the morphisms that factor through an object in the full (replete) subcategory $\Cal Z:=\Cal T\cap\Cal F$ {. Given a pretorsion theory  $(\Cal T, \Cal F)$ we call \emph{trivial}} the objects in $\Cal Z$. Thus the starting idea is  a pair $(\Cal T, \Cal F)$ of {full subcategories} for which the intersection $\Cal Z=\Cal T\cap\Cal F$ replaces the zero object when the base category is not pointed. Under some natural mild assumptions, such a setting already allows us to obtain most of the basic results that are well known for classical torsion theories. { For instance, the torsion-free subcategory $\Cal F$ is epireflective in $\Cal C$, and the torsion subcategory $\Cal T$ is monocoreflective in $\Cal C$.}
Several  new situations can then be studied and understood under this perspective. 

After establishing the basic properties of our pretorsion theories, we have discovered that there is a large overlapping between our results and the results in \cite{GJ} and \cite{GJM}. The main difference is that, in \cite{GJ} and \cite{GJM}, a closed ideal $\Cal N$ of null morphisms of the category $\Cal C$ is fixed first of all, assuming that the identity $1_X\colon X\to X$ of every object $X$ of $\Cal C$ has a kernel and a cokernel with respect to $\Cal N$, and then defining a torsion theory relatively to the ideal $\Cal N$ as a pair $(\Cal T, \Cal F)$ {of full subcategories of $\Cal C$}. In our {approach, we firstly fix the pair of subcategories $(\Cal T, \Cal F)$,}
and then consider kernels and cokernels relatively to the ideal generated by the objects in the intersection $\Cal Z=\Cal T\cap\Cal F$. In this way, we don't need kernels and cokernels of the identity morphisms in the category.
The simplest example of a pretorsion theory that satisfies our conditions, but not those in  \cite{GJ}, is the category {$\Cal C = {\bf 2}$
 obtained from the partially ordered set $\{0,1\}$ with $0<1$. This category has two objects ($0$ and $1$) and three morphisms (the two identities and a unique further morphism $0\to 1$). In this category, there is a pretorsion theory $(\Cal C,\{1\})$ according to our definition (see Example~\ref{good}), but the identity $0\to0$ does not have a $\{1\}$-prekernel. This example shows that, in our approach, the subcategory $\Cal Z$ is not required to be both a reflective and a coreflective subcategory in $\cal C$, differently from \cite{GJ,GJM}. For the rest, the large overlapping with \cite{GJ} is due to the fact that in both papers we try to develop the standard elementary well known elementary properties of torsion theories in abelian categories starting from our weak axioms.

\bigskip

We are grateful to Marco Grandis and Sandra Mantovani for some useful suggestions. 

\section{Basic facts and definitions. Pretorsion theories.}
  
In this section, we give a brief presentation of the results of \cite[Section~4]{AC} we need in this paper. 
  
  \medskip
    
  Let $\Cal C$ be an arbitrary category and $\Cal Z$ be a non-empty class of objects of $\Cal C$. For every pair $A,A'$ of objects of $\Cal C$, we indicate by $\Triv_{\Cal Z}(A, A')$ the set of  all morphisms in $\Cal C$ that factor through an object of $\Cal Z$. We will call these morphisms {\em $\Cal Z$-trivial}.

Let $f\colon A\to A'$ be a morphism in $\Cal C$. We say that a morphism $\varepsilon\colon X\to A$ in $\Cal C $ is a \emph{$\Cal Z$-prekernel} of $f$ if the following properties are satisfied: 
\begin{enumerate}
	\item $f\varepsilon$ is a $\Cal Z$-trivial morphism.
	\item Whenever $\lambda \colon Y\to A$ is a morphism in $\Cal C$ and $f\lambda$ is $\Cal Z$-trivial, then there exists a unique morphism $\lambda'\colon Y\to X$ in $\Cal C$ such that $\lambda=\varepsilon\lambda'$. 
\end{enumerate}

\begin{proposition}\label{prekernel-properties'}{\rm \cite[Proposition 5.1]{AC}} 
Let $f\colon A\to A'$ be a morphism in $\Cal C$ and let $\varepsilon\colon X\to A$ be a $\Cal Z$-prekernel for $f$. Then the following properties hold. 
\begin{enumerate}
	\item[{\rm (a)}] $\varepsilon$ is a monomorphism. 
	\item[{\rm (b)}] If $\lambda\colon Y\to A$ is any other $\Cal Z$-prekernel of $f$, then there exists a unique isomorphism $\lambda'\colon Y\to X$ such that $\lambda=\varepsilon\lambda'$. 
\end{enumerate}
\end{proposition}

(The proof of this Proposition and the other results presented in this Section are simple, and can be found in \cite[Section 4]{AC}.)

Dually, a \emph{$\Cal Z$-precokernel} of $f$ is a morphism $\eta\colon A'\to X$ such that:
\begin{enumerate}
	\item $\eta f$ is a $\Cal Z$-trivial morphism.
	\item Whenever $\mu\colon A'\to Y$ is a morphism and $\mu f$ is $\Cal Z$-trivial, then there exists a unique morphism $\mu'\colon X\to Y$ with $\mu=\mu' \eta$.
\end{enumerate}

\smallskip

If $\Cal C^{\op}$ is the opposite category of $\Cal C$, the $\Cal Z$-precokernel of a morphism $f\colon A\to A'$ in $\Cal C$ is the $\Cal Z$-prekernel of the morphism $f\colon A'\to A$ in $\Cal C^{\op}$. Hence, from Proposition \ref{prekernel-properties'}, we get:

\begin{proposition}\label{precokernel-properties'}
Let $f\colon A\to A'$ be a morphism in a category $\mathcal C$ and let $\eta\colon A'\to X$ be a $\mathcal Z$-precokernel of $f$. Then the following properties hold. 
\begin{enumerate}
	\item[{\rm (a)}]  $\eta$ is an epimorphism.
	\item[{\rm (b)}]  If $\zeta\colon A'\to Y$ is another $\mathcal Z$-precokernel of $f$, then there exists a unique isomorphism $\varphi\colon X\to Y$ satisfying $\zeta=\varphi \eta$. 
\end{enumerate}
\end{proposition}

Let $f\colon A\to B$ and $g\colon B\to C$ be morphisms in $\Cal C$. We say that $$\xymatrix{
	A \ar[r]^f &  B \ar[r]^g &  C}$$ is a \emph{short $\Cal Z$-preexact sequence} in $\Cal C$ if $f$ is a $\Cal Z$-prekernel of $g$ and $g$ is a $\Cal Z$-precokernel of $f$.
	
Clearly, if $\xymatrix{
	A \ar[r]^f &  B \ar[r]^g &  C}$ is a short $\Cal Z$-preexact sequence in $\Cal C$, then\linebreak $\xymatrix{
	C \ar[r]^g &  B \ar[r]^f &  A}$ is a short $\Cal Z$-preexact sequence in $\Cal C^{\op}$.
	
\begin{remark}{\rm 
{ Given any full subcategory $\Cal Z$ of $\Cal C$, the morphisms of $\Cal C$ which factor through $\Cal Z$ is an \emph{ideal of morphism} in the sense of Ehresmann \cite{Eh}. This means that the class of $\Cal Z$-trivial morphisms satisfies the property that the composite 
$$\xymatrix{X \ar[r]^f & Y \ar[r]^g & Z} $$
is $\Cal Z$-trivial whenever one of the two morphisms $f$ or $g$ is $\Cal Z$-trivial.
Categories with an assigned ideal of morphisms were later considered by R.~Lavendhomme \cite{La}, and by M. Grandis  in his foundational work on homological and homotopical algebra \cite{G1,G2,G3}.
More recently, kernels, cokernels and short exact sequences with respect to 
an ideal of morphisms have been considered by various authors. In particular, torsion theories defined with respect to an ideal of morphisms were first introduced by S. Mantovani \cite{Mantovani}, and then also investigated by M.~Grandis and G. Janelidze \cite{GJ}. M. Gran, Z. Janelidze, D. Rodelo and A. Ursini \cite{GJU, GJR} have studied several exactness properties in algebra in categories equipped with an ideal of morphisms.} In this paper, we only consider the ideals generated by the identities of the objects in { the full subcategory} $\Cal Z$ of $\Cal C$.}\end{remark}
	
	The following lemma is particularly important for the rest of the paper ({Cf. \cite[Lemma 4.4]{AC}}):

  \begin{Lemma} \label{xxx}  Let $\Cal C$ be a category, $\Cal Z$ a non-empty class of objects of $\Cal C$ and let $f\colon  A\to B$, $g\colon B\to C$ be  morphisms in $\cal C$. The following statements hold.
  	
  	{\rm (a)} If $g$ is a $\cal Z$-precokernel of $f$, then $f$ is $\Cal Z$-trivial if and only if $g$ is an isomorphism.
  	
  	{\rm (b)} If $f$ is a $\cal Z$-prekernel of $g$, then  $g$ is $\Cal Z$-trivial if and only if $f$ is an isomorphism.\end{Lemma}
	
	\begin{proof} (a) Suppose $f$ $\Cal Z$-trivial. The identity morphism $1_B\colon B\to B$ is clearly a $\Cal Z$-precokernel of $f$. By the uniqueness up to isomorphism of $\Cal Z$-precokernels (Proposition \ref{precokernel-properties'}), there is a unique isomorphism $h\colon C\to B$ such that $hg=1_B$. Hence $g$ is an isomorphism.
	
	Conversely, suppose that $g$ is an isomorphism. Since $gf$ is $\Cal Z$-trivial, i.e., $gf$ factors through an object in $\Cal Z$, the same holds for $f$, that is, $f$ is $\Cal Z$-trivial.

Statement (b) follows by duality. \end{proof}

Another fact that will be freely used in the rest of the paper is the following elementary result. 

\begin{lemma}
\label{exact-iso} Let $\cal C$ be a category, let $\cal Z$ be a non-empty class of objects of $\cal C$ and let 
$$\xymatrix{
	A \ar[r]^f &  B \ar[r]^g &  C}$$ be a short $\Cal Z$-preexact sequence. If $\alpha\colon X\to A$, $\beta\colon C\to Y$ are isomorphisms, then the sequence 
$$\xymatrix{
	X \ar[r]^{f\alpha} &  B \ar[r]^{\beta g} &  Y}$$
is also $\cal Z$-preexact. 
\end{lemma}

  \begin{definition}\label{pretorsion-theory-def} {\rm Let $\Cal C$ be an arbitrary category. A {\em pretorsion theory} $(\Cal T,\Cal F)$ in $\Cal C$ consists of two replete (= closed under isomorphism) full subcategories $\Cal T,\Cal F$  of $\Cal C$, satisfying the following two conditions. Set $\Cal Z:=\Cal T\cap\Cal F$. 
  	
	(1) $\Hom_{\Cal C}(T,F)=\Triv_{\Cal Z}(T, F)$  for every object $T\in\Cal T$, $F\in\Cal F$.}\end{definition}
	
	  (2) For every object $B$ of $\Cal C$ there is a short $\Cal Z$-preexact sequence $$\xymatrix{
	A \ar[r]^f &  B \ar[r]^g &  C}$$ with $A\in\Cal T$ and $C\in\Cal F$.

\bigskip

	In the rest of the paper, whenever we will deal with a pretorsion theory  $(\Cal T,\Cal F)$ for a category $\Cal C$, the symbol $\Cal Z$ will always indicate the intersection $\Cal T\cap\Cal F$. Notice that if $(\Cal T,\Cal F)$ is a pretorsion theory in a category $\Cal C$, then $(\Cal F,\Cal T)$ turns out to be a pretorsion theory in $\Cal C^{\op}$. 
	
	\begin{proposition}\label{ortogonal}{\rm \cite[Proposition 4.5]{AC}} Let $(\Cal T,\Cal F)$ be a pretorsion theory in a category $\Cal C$, and let $X$ be any object in $\Cal C$. \begin{enumerate}
	\item[{\rm (a)}]  If $\Hom_{\Cal C}(X,F)=\Triv_{\Cal Z}(X,F)$ for every $F\in\mathcal F$, then $X\in\Cal T $.
	\item[{\rm (b)}]  If $\Hom_{\Cal C}(T,X)=\Triv_{\Cal Z}(T,X)$ for every $T\in\Cal T $, then $X\in\mathcal F$. 
\end{enumerate}\end{proposition}

As a corollary,  we have that given a pretorsion theory $(\Cal T,\Cal F)$ in a category $\Cal C$, any two of the three classes $\Cal T,\Cal F,\Cal Z$ determine the third one. Moreover:

\begin{corollary}\label{retract} Let $(\Cal T,\Cal F)$ be a pretorsion theory in a category $\Cal C$. Then the three classes $\Cal T$, $\Cal F$ and $\Cal Z$ are all closed under retracts.\end{corollary}

\begin{proof} In order to show that $\Cal T$ is closed under retracts, suppose that $T$ is an object of $\Cal T$, $X$ an object in $\Cal C$, and $f\colon X\to T$ and $g\colon T\to X$ are morphisms in $\Cal C$ with $gf=1_X$. For any object $F$ in $\Cal F$ and morphism $h\colon X\to F$, we have that $h=hgf$, where $hg\in\Hom_{\Cal C}(T,F)=\Triv_{\Cal Z}(T,F)$. Thus $hg$ is $\Cal Z$-trivial, so $h=hgf$ is also $\Cal Z$-trivial. This proves that  $\Hom_{\Cal C}(X,F)=\Triv_{\Cal Z}(X,F)$ for every $F\in\mathcal F$. By Proposition~\ref{ortogonal}, $X\in\Cal T $. This shows that $\Cal T$ is closed under retracts. Dually, $\cal F$ is closed under retracts. Clearly, an intersection of {full subcategories} closed under retracts is closed under retracts.\end{proof}

\begin{remarks}\label{non-empty}
{\rm (a) As a consequence of Axiom (2) in Definition \ref{pretorsion-theory-def}, if $(\Cal T,\Cal F)$ is a pretorsion theory, then, for an object $B$ of $\Cal C$, there is a short $\Cal Z$-preexact sequence $\xymatrix{
	A \ar[r]^f &  B \ar[r]^g &  C.}$ Since the composite morphism $gf$ is $\Cal Z$-trivial, there exists an object in $\Cal Z$. Therefore the intersection $\Cal Z=\Cal T\cap\Cal F$ is always a non-empty class.
	
(b) It is possible that there are objects $B$ in $\Cal C$ such that\linebreak $\Hom_{\Cal C}(Z,B)=\emptyset$ for every $Z\in\Cal Z$ (or the dual: there can exist objects $B$ in $\Cal C$ such that $\Hom_{\Cal C}(B,Z)=\emptyset$ for every $Z\in\Cal Z$.) For instance, an example is given by the pretorsion theory cited in the Introduction, i.e., the category $\Cal C$ obtained from the partially ordered set {$\{0,1\}$ with $0<1$. In this category, for the pretorsion theory $(\Cal C,\{1\})$ (Example~\ref{good}), there is no morphism $Z\to 0$ for $Z\in\Cal Z=\{1\}$.}}
\end{remarks}

\section{First properties of pretorsion theories}

First of all, we prove that the short $\cal Z$-preexact sequence given in Axiom~(2) of Definition \ref{pretorsion-theory-def} is uniquely determined, up to isomorphism. 

\begin{proposition}\label{uniqueness-short-exact-seq}
Let $\cal C$ be a category and let $(\cal T,\cal F)$ be a pretorsion theory for $\cal C$. If 
$$
\xymatrix{
	T \ar[r]^\varepsilon &  A \ar[r]^\eta &  F} \quad\mbox{and}\quad \xymatrix{
	T' \ar[r]^{\varepsilon'} &  A \ar[r]^{\eta'} &  F'}
$$
are $\cal Z$-preexact sequences, where $T,T'\in\cal T$ and $F,F'\in\cal F$, then there exists a unique isomorphism $\alpha\colon T\to T'$ and a unique isomorphism $\sigma\colon F\to F'$ making the diagram 
$$
\xymatrix{
	T \ar[r]^\varepsilon\ar[d]^\alpha &  A \ar[r]^\eta \ar@{=}[d] &  F\ar[d]^\sigma\\ T' \ar[r]^{\varepsilon'} &  A \ar[r]^{\eta'} &  F'}
$$
commute.
\end{proposition}

\begin{proof}
The morphism $\eta'\varepsilon\in \Hom(T,F')$ is $\cal Z$-trivial, because $T\in\cal T$ and $F'\in \cal F$. Since $\varepsilon'$ is a $\cal Z$-prekernel of $\eta'$ there is a unique $\alpha\colon T\to T'$ with $\varepsilon=\varepsilon'\alpha$. Similarly, since $\eta\varepsilon'$ is $\cal Z$-trivial and $\varepsilon$ is a $\cal Z$-prekernel of $\eta$, there is a unique $\beta\colon T'\to T$ with $\varepsilon'=\varepsilon\beta$.  { Since $\varepsilon$ and $\varepsilon '$ are monomorphisms it follows that $\alpha$ is an isomorphism (with inverse $\beta$)}. The proof that there is an isomorphism $\sigma \colon F \rightarrow F'$ is similar.

%Now $\varepsilon=\varepsilon'\alpha=\varepsilon\beta \alpha$ and thus, since $\varepsilon$ is a monomorphism by Proposition~\ref{prekernel-properties'}(1),
%we infer that $\beta\alpha=1_T$. Similarly, from the equality $\varepsilon'=\varepsilon\beta=\varepsilon'\alpha\beta$ and the fact that $\varepsilon'$ is a monomorphism, we get $\alpha\beta=1_{T'}$. This proves that $\alpha$ is an. 
\end{proof}

\begin{example}\label{Car}{\rm As an example, consider an object $B\in\Cal T$ and any short $\Cal Z$-preexact sequence $\xymatrix{
	A \ar[r]^f &  B \ar[r]^g &  C}$ with $A\in\Cal T$ and $C\in\Cal F$. For such a sequence, we have that $g$ is a morphism from an object in $\Cal T$ to an object in $\Cal F$, hence $g$ is $\Cal Z$-trivial. Then $f$ is an isomorphism by Lemma~\ref{xxx}(b).}
\end{example}	
	
Let $\cal C$ be a category and let $(\cal T,\cal F)$ be a pretorsion theory for $\cal C$. By definition and the Axiom of Choice for classes, for every object $X$ of $\cal C$, it is possible to fix a short $\cal Z$-preexact sequence $\xymatrix{
	t(X) \ar[r]^{\varepsilon_X} &  X \ar[r]^{\eta_X} &  f(X),}$
where $t(X)\in\cal T$ and $f(X)\in\cal F$. Moreover, in the choice of such a short $\Cal Z$-preexact sequence for each object $C$ of $\Cal C$, we can assume by Lemma \ref{exact-iso} that:
\begin{enumerate}		
\item  For every object $Z\in \cal Z = \cal T \cap \cal F$,  the chosen short $\cal Z$-preexact sequence is the sequence
	$\xymatrix{
	Z\ar[r]^{1_Z} &  Z \ar[r]^{1_Z} &  Z.}$
\item
For every object $T\in \Cal T\setminus\Cal Z$,  the chosen short $\cal Z$-preexact sequence for $T$ is a sequence of the form 
	$\xymatrix{
	T\ar[r]^{1_T} &  T \ar[r] &  F}$ for an object $F\in\Cal F$.
		\item  For every $F\in \Cal F$,  the chosen short $\cal Z$-preexact sequence is a sequence of the form
$\xymatrix{
	T\ar[r]&F\ar[r]^{1_F} &  F}$ for some $T\in\Cal T$.
\end{enumerate}

Now let $\varphi\colon A\to B$ be a morphism in $\cal C$ and consider the diagram 
\begin{equation}\label{double-naturality}
\xymatrix{
	t(A) \ar@{.>}[d]_{t(\varphi)} \ar[r]^{\varepsilon_A} &  A \ar[r]^{\eta_A}\ar[d]^\varphi &  f(A) \ar@{.>}[d]^{f(\varphi)}. \\t(B) \ar[r]^{\varepsilon_B} &  B \ar[r]^{\eta_B} &  f(B)}\tag{D}
\end{equation}
in which the two rows are the short $\cal Z$-preexact sequences chosen above. Since $\eta_B\varphi\varepsilon_A$ is $\cal Z$-trivial and $\varepsilon_B$ is a $\cal Z$-prekernel of $\eta_B$, there exists a unique (dotted) morphism $t(\varphi)\colon t(A)\to t(B)$ satisfying $\varepsilon_Bt(\varphi)=\varphi\varepsilon_A$. Dually, since $\eta_A$ is a $\cal Z$-precokernel of $\varepsilon_A$, there is a unique (dotted) morphism $f(\varphi)\colon f(A)\to f(B)$ such that $f(\varphi)\eta_A=\eta_B\varphi$. These assignments are clearly functors $t \colon {\Cal C} \rightarrow {\Cal T}$ and $f \colon {\Cal C} \rightarrow {\Cal F}$, and they actually determine the following adjunctions:

\begin{proposition}\label{adjoint}
Let $(\cal T,\cal F)$ be a pretorsion theory in a category $\cal C$. 
Then:
\begin{enumerate}
	\item[{\rm (a)}] The functor $f \colon {\Cal C} \rightarrow {\Cal F}$ is a left-inverse left-adjoint of the (full) embedding $e_{\Cal F}\colon\Cal F \to\Cal C$.		
.
		\item[{\rm (b)}] The functor $t \colon {\Cal C} \rightarrow {\Cal T}$ is a left-inverse right-adjoint of the (full) embedding $e_{\Cal T}\colon\Cal T\to\Cal C$
		\end{enumerate}
\end{proposition}

\begin{proof}
We only need to check (a), the proof of (b) is dual. 

Given an object $A\in\cal C$, consider any morphism $\varphi  \colon A \rightarrow e_{\cal F}(F)$, where $F \in {\Cal F}$.
$$\label{functors}
\xymatrix{
	t(A) \ar[r]^{\varepsilon_A}&  A \ar[r]^{\eta_A}\ar[d]_\varphi &  e_{\cal F} f(A)\ar@{.>}[ld]^{e_{\cal F}(\overline{\varphi})}\\
	 &  e_{\cal F}(F)  &  }
$$
Since $t(A) \in \Cal T$, the composite $\varphi \varepsilon_A$ is $\cal Z$-trivial. By the universal property of the $\cal Z$-precokernel $\eta_A$ of $\varepsilon_A$, there exists a unique morphism $\overline{\varphi} \colon f(A) \rightarrow F$ in $\mathcal F$ such that $e_{\cal F}( \overline{\varphi}) \eta_A = \varphi$. This shows that the functor $f \colon {\Cal C} \rightarrow {\Cal F}$ is left adjoint to $e_{\cal F} \colon  {\Cal F} \rightarrow {\Cal C}$, with $\eta$ the unit of the adjunction (the functor $f$ is thus uniquely determined, up to natural isomorphism). The equality $f\circ e_{\Cal F}=1_{\Cal F}$ follows from the choices (1) and (3) above.
\end{proof}
{
Thanks to Propositions \ref{adjoint}, \ref{prekernel-properties'} and \ref{precokernel-properties'} we immediately see that
\begin{corollary}\label{reflectiveness}
Given any pretorsion theory $(\cal T,\cal F)$ in a category $\cal C$, the torsion-free subcategory $\cal F$ is an epireflective subcategory of $\cal C$
\begin{equation}\label{adj-free}
\xymatrix@=30pt{
{\cal C \, } \ar@<1ex>[r]_-{^{\perp}}^-{f} & {\cal F,}
\ar@<1ex>[l]^-{e_{\cal F}} }
 \end{equation}
i.e., a reflective subcategory with the property that each component of the unit $\eta$ of the adjunction \eqref{adj-free} is an epimorphism.
 Dually, the torsion subcategory $\cal T$ is monocoreflective in $\mathcal C$
 \begin{equation}\label{adj-tor}
\xymatrix@=30pt{
{\cal T \, } \ar@<1ex>[r]_-{^{\perp}}^-{e_{\cal T}} & {\cal C,}
\ar@<1ex>[l]^-{t} }
 \end{equation}
i.e., a coreflective subcategory having the property that each component of the counit $\varepsilon$ of the adjunction \eqref{adj-tor} is a monomorphism.
\end{corollary}
As a consequence, we get that the torsion-free subcategory $\cal F$ is closed under limits in $\mathcal C$, whereas the torsion subcategory $\cal T$ is closed under colimits in $\mathcal C$. This extends the classical situation of the abelian setting \cite{Dickson} and of the homological one \cite{BG}.
 From $t\circ e_{\Cal T}=1_{\Cal T}$ and $f\circ e_{\Cal F}=1_{\Cal F}$, 
we get that $E_{\Cal T}:=e_{\Cal T}\circ t$ and $E_{\Cal F}:=e_{\Cal F}\circ f$ are two idempotent functors $\Cal C\to\Cal C$
		with natural transformations $\varepsilon\colon E_{\Cal T}\to 1_{\Cal C}$ and $\eta\colon 1_{\Cal C}\to E_{\Cal F}$.
		The idempotent functor $E_{\Cal T}\colon\Cal C\to\Cal C$ is a subfunctor of the identity functor on $\Cal C$ (by Proposition \ref{prekernel-properties'}).}
\begin{lemma}\label{utilissimo}
Let $\cal Z\subseteq \cal F$ be full subcategories of a category $ \cal C$. Assume that $\cal Z$ is closed under retracts in $\cal C$ and that $\cal F$ is reflective in $\cal C$. If $T$ is an object of $\cal C$ and the $T$-component $\eta_T:T\to f(T)$ of the unit $\eta$ of the adjunction 
\begin{equation*}
\xymatrix@=30pt{
	{\cal C \, } \ar@<1ex>[r]_-{^{\perp}}^-{f} & {\cal F}
	\ar@<1ex>[l]^-{e_{\cal F}} }
\end{equation*}
is $\cal Z$-trivial, then $f(T)\in \cal Z$.
\end{lemma}
\begin{proof}
By assumption there is a commutative triangle $$
\xymatrix{T \ar[rr]^{\eta_T} \ar[dr]_\alpha& & f(T) \\
	&Z  \ar[ur]_\beta  & 
}
$$
where $Z \in \Cal Z$.  The universal property of $\eta_T$ and the fact that $Z \in \Cal F$ imply that there exists a unique $\gamma \colon f(T) \rightarrow Z$ such that $\gamma \eta_T = \alpha$. It follows that $$ \beta \gamma \eta_T = \beta \alpha = \eta_T = 1_{f(T)} \eta _T,$$
and $ \beta \gamma = 1_{f(T)}$. Since $\cal Z$ is closed in $\mathcal C$ under retracts, we conclude that $f(T)\in \cal Z$.
\end{proof}	
		\begin{proposition}\label{importante!} { The following conditions hold in any pretorsion theory $(\Cal T,\cal F)$ in a category $\Cal C$:
	\begin{enumerate}
	\item[{\rm (a)}]
	$f(T)\in\Cal Z$  for every $T\in\cal T$.
	\item[{\rm (b)}]  For every $T\in\cal T$, there exists a short $\Cal Z$-preexact sequence of the type $$\xymatrix{
	T \ar[r]^{1_T} &  T \ar[r]^{\eta_T} &  Z}$$ with $Z\in\Cal Z$.
	%\item[{\rm (c)}]  For every $T\in\cal T$, the identity morphism $1_T\colon T\to T$ has a $\cal Z$-precokernel $\eta\colon T\to Z$ with $Z\in\Cal Z$.
	\item[{\rm (c)}] The restriction $\hat{f} \colon\Cal T\to\cal Z$ of $f$ is the left adjoint of the embedding $\Cal Z\to\Cal T$. 
	\end{enumerate} }
	\end{proposition}
	
	\begin{proof} 
	$(a)$ For every $T\in\cal T$, the morphism $\eta_T:T\to f(T)$ is $\cal Z$-trivial, since $T\in \cal T$ and $f(T)\in \cal F$. Moreover, by Corollary \ref{retract}, $\cal Z$ is closed under retracts. Then apply Lemma \ref{utilissimo} to get the conclusion. 
	
{(b) For any $T \in \cal T$ there is a short $\cal Z$-preexact sequence $$\xymatrix{t(T) \ar[r]^f & T \ar[r]^-\eta & f(T) }$$ with $f(T)  \in \cal Z$  by (a). The morphism $f$ is an isomorphism, and the sequence required in (b) is the short $\cal Z$-preexact sequence $\xymatrix{T \ar@{=}[r] & T \ar[r]^\eta & f(T)  }$.
	
	%(c) Any morphism $\eta\colon T\to Z$ with $Z\in\Cal Z$ has the identity $1_T\colon T\to T$ as its $\Cal Z$-prekernel. Since (b) holds, we have that $Z\cong f(T)$. Thus $f(T)\in\Cal Z$.
	
	(c) If $f(T)\in\Cal Z$  for every $T\in\cal T$, we can restrict the functor $f\colon \Cal C\to\Cal F$ to the functor {$\hat{f} \colon \Cal T\to\Cal Z$, which is }then the { left}  adjoint of the inclusion $\cal Z\to\cal T$. 
	}
	\end{proof}

In particular, we have the following 
		\begin{proposition}\label{caratterizzazione} Let $\cal C$ be a category and let $(\cal T,\cal F)$ be a pretorsion theory for $\cal C$. The following conditions are equivalent for an object $C$ of $\Cal C$:\begin{enumerate}\item[{\rm (a)}] $C\in\Cal T$. 
	\item[{\rm (b)}] $C\cong t(C)$.
	\item[{\rm (c)}] $C\cong t(X)$ for some object $X$ in $\Cal C$.
\end{enumerate}\end{proposition}

%\begin{proof} (a)${}\Rightarrow{}$(b) Suppose $C\in\Cal T$ and consider the short $\cal Z$-preexact sequence 
%	$$\xymatrix{
%		t(C) \ar[r]^{\varepsilon_C} &  C \ar[r]^{\eta_C} &  f(C)}.$$
%{	The morphism $\eta_C$ is then $\cal Z$-trivial, thus $\varepsilon_C$ is an isomorphism, by Lemma \ref{xxx}(b). }
%(b)${}\Rightarrow{}$(c) and (c)${}\Rightarrow{}$(a) are trivial, because $\cal T$ is replete.
	%\end{proof}
	
Dually, we have:

\begin{proposition}
Let $\cal C$ be a category and let $(\cal T,\cal F)$ be a pretorsion theory for $\cal C$. The following conditions are equivalent for an object $C $ of $\cal C$: 
\begin{enumerate}
	\item[{\rm (a)}]  $C\in \cal F$.
	\item[{\rm (b)}]   $C\cong f(C)$.
	\item[{\rm (c)}]  $C\cong f(X)$ for some object $X$ in $\cal C$. 
\end{enumerate}
\end{proposition}
 Let $\cal Z\subseteq \cal F$ be full subcategories of a category $\cal C$. Set
$$
\cal F^{\perp_{\cal Z}}:=\{T\in \cal C\mid \Hom_{\Cal C}(T,F)=\Triv_{\Cal Z}(T,F), \mbox{for every }F\in \cal F \}.
$$
\begin{proposition}[{Cf. \cite[Lemma 4.7]{BG}}]
Let $\cal Z\subseteq \cal F$ be full subcategories of a category $\cal C$. Assume that $\cal Z$ is closed under retracts in $\cal C$ and that $\cal F$ is reflective in $\cal C$. Let $f:\cal C\to \cal F$ denote the left adjoint of the inclusion  $\cal F\subseteq \cal C$. Then
$$
\cal F^{\perp_\cal Z}=\{T\in \cal C\mid f(T)\in \cal Z \}.
$$
\end{proposition}
\begin{proof}
For any object $A$ in $\cal C$, let $\eta_A:A\to f(A)$ be the $A$-component of the unit of the adjunction. First take an object $T$ in $\cal C$ such that $f(T)\in \cal Z$. Given an object $F\in \cal F$ and any morphism $\alpha:T\to F$, there is a unique morphism $\overline{\alpha}:f(T)\to F$ satisfying $\alpha=\overline{\alpha}\eta_T$. Since $f(T)\in \cal Z$, it follows immediately that $\alpha$ is $\cal Z$-trivial. Conversely, take an object $T\in \cal F^{\perp_{\cal Z}}$. In particular, the morphism $\eta_T:T\to f(T)$ is $\cal Z$-trivial. Since $\cal Z$ is closed under retracts, it suffices to apply Lemma \ref{utilissimo} to get the conclusion. 
\end{proof}

The classical identity $t(A/t(A))=0$ for { the radical $t$ associated with a torsion theory in an abelian category} now becomes:

\begin{proposition}\label{noncommutano}  Let $(\cal T, \cal F)$ be a pretorsion theory in $\mathcal C$. There is a natural transformation $\zeta\colon E_{\Cal F}E_{\Cal T}\to E_{\Cal T}E_{\Cal F}$. Moreover $E_{\Cal F}E_{\Cal T}(C)$ and $E_{\Cal T}E_{\Cal F}(C)$ are objects of $\Cal Z$ for every $C$ in $\Cal C$. \end{proposition}

\begin{proof} From Proposition~\ref{importante!}(a), we know that $f(t(C))\in\Cal Z$, and from its dual we have that $t(f(C))\in\Cal Z$.
	We must show that for every object $C$ of $\Cal C$ there is a natural morphism $\zeta_C\colon f(t(C))\to t(f(C))$. Applying the functor $t$ to the morphism $\eta_C\colon C\to f(C)$, we get a morphism $t(\eta_C)\colon t(C)\to t(f(C))$. This is a $\Cal Z$-trivial morphism, because $t(f(C))\in\Cal Z$ (by the dual version of Proposition \ref{importante!}(a)). Moreover, the morphism $\eta_{t(C)}\colon t(C) \to f(t(C))$ is the $\Cal Z$-precokernel of the identity $t(C)\to t(C)$. Hence there is a unique morphism $\zeta_C\colon f(t(C))\to t(f(C))$ such that $t(\eta_C)=\zeta_C\circ \eta_{t(C)}$. 
	
	To prove the naturality of $\zeta$, let $\varphi\colon C\to C'$ be a morphism in $\Cal C$, and consider the diagram
	\begin{equation} \xymatrix{
		ft(C) \ar[ddd]_{ft(\varphi)} \ar[rr]^{\zeta_C}& &tf(C)\ar[ddd]^{tf(\varphi)}\\
		& t(C)\ar[ul]^{\eta_{t(C)}}\ar[ur]_{t(\eta_C)}\ar[d]^{t(\varphi)} &\\
		& t(C')\ar[dl]_{\eta_{t(C')}}\ar[dr]^{t(\eta_{C'})}&\\
		ft(C')\ar[rr]_{\zeta_{C'}}&&tf(C'),}\label{44}\end{equation} where:
	
	(a) The trapezoids on the left and on the right commute by the naturality of $\eta$ and the functoriality of $t$.
	
	(b) The upper triangle and the lower triangle commute by the definition of $\zeta$.
	
	(c) $\eta_{t(C)}$ is an epimorphism.
	
	\noindent It follows that the outer rectangle in diagram (\ref{44}) commutes. This proves that the transformation $\zeta$ is natural.\end{proof}

\begin{example}{\rm 	Here is an example of a pretorsion theory for which the natural transformation $\zeta$ of Proposition~\ref{noncommutano} is not an isomorphism.
		Consider the non-modular lattice $N_5$: 
		%	\begin{equation*} \xymatrix{ & F\ar@{-}[rd]\ar@{-}[ldd] & \\ 
		%	&& {Z'}\ar@{-}[dd] \\
		%	C\ar@{-}[rdd] & &\\
		%	&& {Z}\ar@{-}[ld] \\
		%		&T&}\end{equation*} 

		
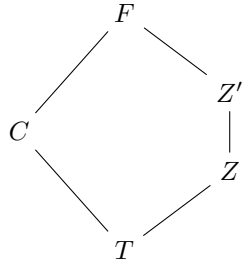
\begin{center}\begin{tikzpicture}[scale=.7]
			\node (one) at (0,2) {$F$};
			\node (a) at (-2,-0.25) {$C$};
			\node (b) at (2,0.5) {$Z'$};
			\node (c) at (2,-1) {$Z$};
			\node (d) at (0,-2.5) {$T$};
			\draw (a) -- (one) -- (b) -- (c) -- (d) -- (a);
			\end{tikzpicture}\end{center}
		
		The partially ordered set $N_5$ { is} a category with five objects, for which for any pair of objects $X,Y$ there is at most one morphism $X\to Y$, and such a morphism exists if and only if $X\le Y$. We will denote this unique morphism by $X\le Y$. In this category, $T$ is the initial object, $F$ is the terminal object, and two objects are isomorphic if and only if they are equal. Consider the pretorsion theory $(\Cal T,\Cal F)$ for which $\Cal T=\{T,Z,Z'\}$ and $\Cal F=\{F,Z,Z'\}$, so $\Cal Z=\{Z,Z'\}$. It is easy to check that this is a pretorsion theory. For instance, the short $\Cal Z$ preexact sequences for the five objects of the category are $T\le T\le Z$, $Z\le Z\le Z$, $Z'\le Z'\le Z'$, $T\le C\le F$ and $Z'\le F\le F$. Hence $ft(C)=Z$ and $tf(C)=Z'$, so that $ft(C)$ and $tf(C)$ are not isomorphic.}\end{example}

We already know from Proposition~\ref{caratterizzazione} that $\Cal T=\{\,C\in\Cal C\mid C\cong t(C)\,\}$, i.e., $\Cal T=\{\,C\in\Cal C\mid C\cong E_{\Cal T}(C)\,\}$. By duality, $\Cal F=\{\,C\in\Cal C\mid C\cong E_{\Cal F}(C)\,\}$. Also, $\Cal Z=\{\,C\in\Cal C\mid C\cong t(C)\cong f(C)\,\}$. Now we have that:

\begin{proposition} Let $(\cal T, \cal F)$ be a pretorsion theory in $\mathcal C$. Then
	
	$$\Cal Z=\{\,C\in\Cal C\mid C\cong E_{\Cal T}E_{\Cal F}(C)\,\}.$$ \end{proposition}

\begin{proof} If $C\cong E_{\Cal T}E_{\Cal F}(C)$, we have $C\in\Cal Z$ by Proposition~\ref{noncommutano}. Conversely, if $Z\in\Cal Z$, the short $\cal Z$-preexact sequence relative to $Z$ is the sequence
	$$\xymatrix{
		Z\ar[r]^{1_Z} &  Z \ar[r]^{1_Z} &  Z.}$$
	Hence $f(Z)=Z$, so {$E_{\Cal F}(Z)=Z$}, and therefore $Z=t(Z)=E_{\Cal T}(Z)=E_{\Cal T}E_{\Cal F}(Z)$.\end{proof}

\section{Closure properties}

\begin{definition}
{\rm Let $\cal C$ be a category and $\cal Z$ a non-empty full subcategory of $\cal C$. We say that a full replete subcategory $\cal S$ of $\cal C$ is \emph{closed under $\cal Z$-extensions} if, for every short $\cal Z$-exact sequence 
$
S_1\to X\to S_2
$
in $\cal C$, where $S_1,S_2\in\cal S$, then $X\in\cal S$. }
\end{definition}

\begin{proposition}\label{ext}
Let $(\cal T,\cal F)$ be a pretorsion theory in a category $\cal C$. Then the three subcategories $\cal T$, $\cal F$ and $\Cal Z$ are all closed under $\cal Z$-extensions.
\end{proposition}  

\begin{proof}
Fix any object $X$ of $\cal C$ and a short $\cal Z$-preexact sequence $$\xymatrix{
	t(X) \ar[r]^{k} &  X \ar[r]^{p} &  f(X),}$$
where $t(X)\in\cal T$ and $f(X)\in \cal F$. Let  $\xymatrix{
	T_1 \ar[r]^f &  X \ar[r]^g &  T_2}$ be a short $\cal Z$-preexact sequence in $\cal C$, where $T_1$ and $T_2$ belong to $\cal T$. Since $pf$ is $\cal Z$-trivial and $g$ is a $\cal Z$-precokernel of $f$, then there exists a unique morphism $\gamma
	\colon T_2\to f(X)$ such that $p=\gamma g$. By definition, $\gamma$ is $\cal Z$-trivial and thus $p$ is $\cal Z$-trivial too, a fortiori.  Lemma \ref{xxx}(b) implies that $k$ is an isomorphism. Since $\cal T$ is replete, it follows that $X\in\cal T$. Dually, $\cal F$ is closed under $\cal Z$-extensions. The assertion for $\Cal Z$  follows immediately from the fact that $\cal Z = \cal F \cap \cal T$.
\end{proof}

\begin{proposition}
Let $\cal C$ be a category, $(\cal T,\cal F)$ be a pretorsion theory in a category $\cal C$, and $\Cal Z=\cal T\cap \cal F$. 
\begin{enumerate}
	\item[{\rm (a)}]  Assume that $\cal Z$ is closed { under coproducts in $\cal C$. If $\cal X = \{X_j\}_{j\in J}$ is a family of objects of $\cal T$ and there exists a coproduct $X$ of $\cal X$ in $\cal C$, then $ X\in \cal T$. }
	\item[{\rm (b)}]  Assume that $\cal Z$ is closed { under products in $\cal C$. If $\cal Y = \{Y_j \}_{j\in J}$ is a family of objects of $\cal F$ and there exists a product $Y$ of $\cal Y$ in $\cal C$, then $Y \in \cal F$. }
\end{enumerate}
\end{proposition}

\begin{proof}
(a)  Let $\{X_j \}_{j\in J}$ be a family of objects in $\mathcal T$, and write $i_j\colon X_j\to  X$ for the canonical morphism to the (object part of their) coproduct, for any $j \in J$. Let $\alpha\colon  X\to F$ be any morphism with $F \in \cal F$, and it will suffice to show that $\alpha$ is $\cal Z$-trivial (by Proposition \ref{ortogonal}).
Since $\cal X\subseteq \cal T$, the morphism $\alpha i_j\colon X_j\to F$ is $\cal Z$-trivial, for every $j\in J$, that is, there exists an object $Z_j$ of $\cal Z$ and morphisms $\varphi_j\colon X_j\to Z_j$, $f_j\colon Z_j\to F$ such that $\alpha i_j=f_j\varphi_j$. By assumption, there exists a coproduct $Z$ of the family $\{Z_j\}_{j\in J }$ of objects of $\cal Z$, and it belongs to $\cal Z$. 
If $l_j\colon Z_j\to Z$ is the canonical morphism, there is a unique morphism $\varphi\colon  X\to Z$ such that $\varphi i_j=l_j\varphi_j$. The unique morphism $f
\colon Z\to F$ such that $f_j=fl_j$ (for every $j\in J$) clearly satisfies $f\varphi=\alpha$, and this shows that $\alpha$ is $\cal Z$-trivial.
%$$
%f\varphi i_j=f l_j\varphi_j=f_j\varphi_j=\alpha i_j,
%$$
%proving that  and $\alpha$ is $\cal Z$-trivial. 

Statement (b) clearly holds, by duality.
\end{proof}

Let us consider full replete subcategories $\cal Z\subseteq \cal F$ of a category $\cal C$. We say that $\cal F$ is \emph{$\cal Z$-normal epireflective in $\cal C$} if the inclusion functor $e_{\cal F} \colon \cal F \rightarrow \cal C$ has a left adjoint $f \colon \cal C \rightarrow \cal F$ with the property that each component $\eta_A  \colon A \rightarrow f(A)$ of the unit of the adjunction is a $\cal Z$-precokernel. 

\begin{lemma}\label{coker-of-kernel}
Let $\cal C$ be a category and let $\cal Z$ be a full subcategory of $\cal C$. If $p$ is a $\cal Z$-precokernel of some morphism, then $p$ is also the $\cal Z$-precokernel of its $\cal Z$-prekernel $e:A\to B$ whenever the latter exists. In this case, $\xymatrix{
	A \ar[r]^e &  B \ar[r]^p &  C}$ is a short $\cal Z$-preexact sequence in $\cal C$. 
\end{lemma}

\begin{proof}
Let  $q\colon X\to B$ be a morphism such that $p$ is a $\cal Z$-precokernel of $q$. Since $pq$ is $\cal Z$-trivial and $e \colon A \rightarrow B$ is a $\cal Z$-prekernel of $p$, there is a unique morphism $\alpha\colon X\to A$ such that $q=e\alpha$. Now let $\lambda\colon B\to Y$ be any morphism with $\lambda e$ $\cal Z$-trivial. A fortiori, $\lambda q=\lambda e\alpha$ is $\cal Z$-trivial. Since $p$ is a $\cal Z$-precokernel of $q$, there exists a unique morphism $\lambda_1\colon C\to Y$ satisfying $\lambda=\lambda_1p$, as desired.  
\end{proof}

 \begin{proposition}\label{idempotent-subfunctor}
Let $\cal C$ be a category and let $\cal Z$ be a non-empty full subcategory of $\cal C$ closed under retracts. The following conditions on a full replete subcategory $\cal F$ of $\cal C$ are equivalent:
\begin{enumerate}
\item[{\rm (a)}]  $\cal F$ is the torsion-free subcategory of a pretorsion theory $(\cal T, \cal F)$ in $\cal C$ such that  $\cal Z = \cal T \cap \cal F$.
\item[{\rm (b)}]  $\cal F$ is $\cal Z$-normal epireflective in $\cal C$; for any $A\in \cal C$, the $A$-component $\eta_A \colon A \rightarrow f(A)$  of the unit has a $\cal Z$-prekernel $\varepsilon_A \colon t(A) \rightarrow A$; and $\varepsilon_{t(A)} \colon t(t(A)) \rightarrow t(A)$ is an isomorphism for every $A\in \cal C$.
\end{enumerate}
\end{proposition}

\begin{proof}
(a)${} \Rightarrow{} $(b) follows from Propositions \ref{adjoint} and \ref{caratterizzazione}.

(b)${} \Rightarrow{} $(a) First note that, since $\cal F$ is $\cal Z$-normal epireflective, Lemma \ref{coker-of-kernel} implies that, for every object $A$ of $\cal C$, the sequence of morphisms 
$$
\xymatrix{
	t(A) \ar[r]^{\varepsilon_A} &  A \ar[r]^{\eta_A} &  f(A)}
$$
is $\cal Z$-preexact. 

Now define $\cal T$ to be the full subcategory of $\cal C$ whose objects $A$ are those for which $\eta_A \colon A \rightarrow f(A)$ is $\cal Z$-trivial. By definition $\cal Z\subseteq \cal T$. For every $Z\in \cal Z$, since $\varepsilon_Z\colon t(Z)\to Z$ is $\cal Z$-trivial and $
\xymatrix{
	t(Z) \ar[r]^{\varepsilon_Z} &  Z \ar[r]^{\eta_Z} &  f(Z)}
$
is $\cal Z$-preexact, we infer that $\eta_Z$ is an isomorphism, in view of Lemma \ref{xxx}, and thus $Z\in\cal F$, because $\cal F$ is replete. It follows that $\cal Z\subseteq \cal T\cap \cal F$. Conversely, if $C \in \cal T \cap \cal F$, there is no restriction in assuming that the unit $\eta_C$ is the identity on $C$ (since $C \in \cal F$), and $\eta_C$ is $\cal Z$-trivial (because $C \in \cal T$). But $\cal Z$ is closed under retracts, so {$C\in \cal Z$}. This proves that $\cal Z=\cal T\cap \cal F$. 

As we have seen in the first part of the proof of (b)${} \Rightarrow {}$(a), for every object $A$ of $\cal C$, the sequence of morphisms 

$$
\xymatrix{
	t(t(A)) \ar[r]^{\varepsilon_{t(A)}} &  t(A) \ar[r]^{\eta_{t(A)}} &  f(t(A))}
$$
is $\cal Z$-preexact. Since, by assumption, $\varepsilon_{t(A)}$ is an isomorphism, Lemma \ref{xxx} implies that $\eta_{t(A)}$ is $\cal Z$-trivial. This proves that $t(A)$ is in $\cal T$ for every $A\in \cal C$. 

Finally, let $u\colon T\to F$ be a morphism, where $T\in \cal T$ and $F\in \cal F$. 
Then there is a unique morphism $\overline{u }\colon f(T)\to F$ such that $u=\overline{u }\eta_T$ (since $F\in\cal F$). The fact that $T\in\cal T$ means that $\eta_T$ is $\cal Z$-trivial and thus, a fortiori, $u$ is $\cal Z$-trivial. Therefore $(\cal T,\cal F)$ is a pretorsion theory in $\cal C$. 
\end{proof}

\section{Projective objects}

%The next goal is to provide sufficient conditions in order to satisfy the equivalent properties of Proposition \ref{importante!}. 
In our setting, by a \emph{projective object} we shall always mean \emph{projective with respect to epimorphisms}.
%We start from a general remark. 

\begin{proposition}\label{proj-reg-epi}
Let $\cal C$ be a category and let $\cal Z$ be a non-empty { full subcategory of $\cal C$} consisting of projective objects. If $f\colon A\to B$ is any morphism, then any $\cal Z$-precokernel of $f$ (if it exists) is a regular epimorphism.
\end{proposition}

\begin{proof}
Assume that there exists a $\cal Z$-precokernel $p\colon B\to C$ of $f$. By definition, $pf$ is $\cal Z$-trivial, and thus we can pick an object $Z\in \cal Z$ and morphisms $\alpha\colon A\to Z$, $\beta\colon Z\to C$ such that $\beta\alpha=pf$. Since, by assumption, $Z$ is projective and $p$ is an epimorphism (see Proposition \ref{precokernel-properties'}(1)), there exists a morphism $\overline{\beta}\colon Z\to B$ such  that $\beta=p\overline{\beta}$. To show that $p$ is a regular epimorphism, it suffices to prove that $p$ is the coequalizer of $f,\overline{\beta}\alpha\colon A\to B$. Clearly, $p\overline{\beta}\alpha=\beta\alpha=pf$. Now take any morphism $q\colon B\to C'$ such that $qf=q\overline{\beta} \alpha$. Then $qf$ is $\cal Z$-trivial, because  it is equal to $q\overline{\beta} \alpha$. Since $p$ is a $\cal Z$-precokernel of $f$, there exists a unique morphism $q_1\colon C\to C'$ satisfying $q=q_1p$. This allows us to conclude. 
\end{proof}

\begin{remark}
\emph{The assumption on the full subcategory $\cal Z$ of $\cal C$ in Proposition \ref{proj-reg-epi} breaks the duality of the setting. This assumption implies that any $\cal Z$-precokernel is a regular epimorphism, not that any $\cal Z$-prekernel is a regular monomorphism.}
\end{remark}

{ Recall that an \emph{extremal epimorphism} is a morphism $f \colon X \rightarrow Y$ with the following property:
whenever there is a commutative triangle 
$$
\xymatrix{X \ar[rr]^f \ar[dr]_u &&  Y \\
& I \ar[ur]_m & 
}
$$
with $m$ a monomorphism, then $m$ is an isomorphism. Note that if a composite $g f$ is an extremal epimorphism, then $g$ is an extremal epimorphism.

\begin{corollary}\label{stability-quotients}
Let $(\cal T,\cal F)$ be a pretorsion theory in a category $\cal C$. Suppose that $ \cal Z = \cal T \cap \cal F$ consists of projective objects. Then:\begin{enumerate}
\item[{\rm (a)}]  the subcategory $\cal T$ is closed under extremal quotients;
\item[{\rm (b)}]  the subcategory $\cal F$ is closed under subobjects.
\end{enumerate}
\end{corollary}

\begin{proof}
(a) Consider an extremal epimorphism $g \colon T \rightarrow C$, where $T \in \cal T$. We have a commutative diagram
$$
\xymatrix{ T \ar@{=}[r] \ar[d]_{t(g)} & \ar[d]  T \ar[r]^{\eta_T} \ar[d]_g & f(T) \ar[d]^{f(g)} \\
T(C) \ar[r]_-{\varepsilon_C}  & C \ar[r]_{\eta_C} & f(C), 
}
$$
where $\varepsilon_C  t(g) = g$ is an extremal epimorphism. Since $\varepsilon_C$ is a monomorphism (by Proposition \ref{prekernel-properties'}), it is then an isomorphism. Thus $\varepsilon_C$ is an isomorphism, and $C\in \cal T$.

\noindent (b) Let $n \colon N \rightarrow F$ be a monomorphism with $F \in \cal F$. Then $n$ factors through the reflection $\eta_N \colon N \rightarrow f(N)$, i.e., there is a unique $\varphi \colon f(N) \rightarrow F$ with $\varphi \eta_N = n$. This implies that $\eta_N$ is a monomorphism. Since it is also a regular epimorphism (by Proposition \ref{proj-reg-epi}), it follows that $\eta_N$ is an isomorphism, and $N \in \cal F$.
\end{proof}
}

	\begin{corollary}
		Let $\cal C$ be a category in which any arrow $f$ has a factorization $f=mp$, where $p$ is an extremal epimorphism and $m$ is a monomorphism. Let $\cal T,\cal F$ be full replete subcategories of $\cal C$ such that $\cal Z:=\cal T\cap \cal F$ consists of projective objects, and assume that for every object $X$ in $\cal C$ there exists a short $\cal Z$-preexact sequence $T\to X\to F$, for some $T\in\cal T$ and $F\in\cal F$. Then the following conditions are equivalent. 
		\begin{enumerate}
			\item[\rm (a)] $(\cal T, \cal F)$ is a pretorsion theory.
			\item[\rm (b)] $\cal T$ is closed under extremal quotients and $\cal F$ is closed under subobjects. 
		\end{enumerate} 
	\end{corollary}
	\begin{proof}
		The fact that (a) implies (b) is clear, by Corollary \ref{stability-quotients}. Conversely, take objects  $T\in \cal T,F\in \cal F$ and any morphism $f:T\to F$. By assumption, there are an extremal epimorphism $p:T\to Z$ and a monomorphism $m:Z\to F$ such that $f=mp$. By condition (b), $Z\in \cal T\cap \cal F=\cal Z$, being it an extremal quotient of $T$ and a subobject of $F$. Thus $f$ is $\cal Z$-trivial. The conclusion follows. 
	\end{proof}

There are situations where $\cal T \cap \cal F = \cal Z$ is precisely the full subcategory of projective objects in $\cal C$, as in Example \ref{Examples}.1 of the pretorsion theory $(\Equiv,\ParOrd)$ in the category $\pre$ of preordered sets, where $\Equiv$ is the category of equivalence relations and $\ParOrd$ the category of partially ordered sets (see \cite{AC} and Example~\ref{6.1} for more details). 

{ Note that, since projective objects are always closed under retracts, Proposition \ref{idempotent-subfunctor} gives the following:}

\begin{corollary}\label{all-projectives}
Let $\cal C$ be a category. Suppose that $\cal Z$ is the non-empty full subcategory of $\cal C$ that consists of all the projective objects in $\cal C$. The following conditions are equivalent:
\begin{enumerate}
\item[{\rm (a)}] $\cal F$ is the torsion-free subcategory of a pretorsion theory $(\cal T, \cal F)$ in $\cal C$, with $\cal Z = \cal T \cap \cal F$.
\item[{\rm (b)}] $\cal F$ is $\cal Z$-normal epireflective in $\cal C$, for any $A\in \cal C$ the $A$-component $\eta_A \colon A \rightarrow f(A)$  of the unit has a $\cal Z$-prekernel $\varepsilon_A \colon t(A) \rightarrow A$, and $\varepsilon_{t(A)} \colon t(t(A)) \rightarrow t(A)$ is an isomorphism.
\end{enumerate}

\end{corollary}

  \section{Examples}\label{Examples}

 \subsection{The category of preordered sets}\label{6.1}
  In \cite{AC}, the following example was studied in detail. Let $\pre$ be the category of preordered sets. Its objects are all pairs $(A,\rho)$, where $A$ is a set and $\rho$ is a preorder on $A$, that is, a relation on $A$ that is reflexive and transitive. The pretorsion is the pair $(\Equiv,\ParOrd)$, where $\Equiv$ consists of all $(A,\rho)$, where $A$ is a set and $\rho$ is an equivalence relation on $A$, that is, a preorder that is also symmetric, and $\ParOrd$ consists of all $(A,\rho)$, where $A$ is a set and $\rho$ is a partial order on $A$, that is, a preorder that is also antisymmetric. It is easily seen that in the category $\pre$: epimorphisms are exactly the morphisms that are surjective mappings; monomorphisms are exactly the morphisms that are injective mappings; an object of $\pre$ is projective if and only if it is a trivial object, that is, an object in $\cal Z =\Equiv\cap\ParOrd$. %Note that the projective objects are precisely the \emph{discrete} equivalence relations.} 
  An object of $\pre$ is injective if and only if it is of the form $(X,\omega)$, where $X$ is any set and $\omega:=X\times X$ is the \emph{trivial} equivalence relation on $X$, that is, $x\omega y$ for every $x,y\in X$. A generalization of this example of pretorsion theory has been studied in \cite{FFG}, where the category of preordered sets is replaced by any category of \emph{internal preorders in an exact category}.

\subsection{The category of endomappings of finite sets} A related example is given in \cite{AL}\label{8.2}. Let $\Cal M$ be the category whose objects are all pairs $(X,f)$, where $X=\{1,2,3,\dots,n\}$ for some non-negative integer $n$ and $f\colon X\to X$ is a mapping. Notice that, for $n=0$, $X$ is the empty set. The morphisms in $\Cal M$ from $(X,f)$ to $(X',f')$ are the mappings $\varphi\colon X\to X'$ such that $f'\varphi=\varphi f$. The pretorsion theory in $\Cal M$ is the pair $(\Cal T,\Cal F)$, where $\Cal T$ consists of all objects $(X,f)$ of $\Cal M$ with $f\colon X\to X$ a bijection, and $\Cal F$  consists of all objects $(X,f)$ of $\Cal M$ with $f^n=f^{n+1}$, where $n=|X|$. The trivial objects, that is, the objects in $\Cal Z:=\Cal T\cap\Cal F$ are the pairs $(X,1_X)$, where $1_X\colon X\to X$ is the identity mapping.

There is a functor $U\colon \Cal M\to\pre$, which is a canonical embedding \cite{AL}. It associates to any object $(X,f)$ of $\Cal M$ the preordered set $(X,\rho_f)$, where $\rho_f$ is the relation on $X$ defined, for every $x,y\in X$, by $x\rho_f y$ if $x=f^t(y)$ for some integer $t\ge0$. Any morphism $\varphi\colon (X,f)\to (X',f')$ in $\Cal M$ is a morphism $\varphi\colon (X,\rho_f)\to (X',\rho_{f'})$ in $\pre$, because if $x,y\in X$ and $x\rho_f y$, then $x=f^t(y)$ for some integer $t\ge0$. From the equality $f'\varphi=\varphi f$, we get that $\varphi(x)=\varphi f^t(y)={f'}^t\varphi(y)$, so $\varphi(x)\rho_{f'} \varphi(y)$.

It is possible to view $\Cal M$ as a subcategory of $\pre$ (via $U$). Via this identification, one has that $\Cal T =\Cal M\cap\Equiv$ and $\Cal F =\Cal M\cap\ParOrd$ \cite{AL}.

The only projective trivial object of $\Cal M$, that is, the only object of $\Cal Z$ that is projective in $\Cal M$, is the empty set. To see it, consider any object $(X,1_X)$ in $\Cal Z$ with $X\ne\emptyset$. Consider the object $(\{1,2\},(1\ 2))$ (set of two elements $1,\ 2$ with the transposition $(1\ 2)$ of those two elements) of $\Cal M$. Then there are no morphisms $(X,1_X)\to(\{1,2\},(1\ 2))$. Therefore the unique (constant) mappings $\alpha\colon X\to \{1\}$ and $\beta\colon \{1,2\}\to \{1\}$, are surjective morphisms in $\Cal M$, hence they are epimorphisms in $\Cal M$, but there is no morphism $\gamma\colon (X,1_X)\to(\{1,2\},(1\ 2))$ such that $\beta\gamma=\alpha$. This shows that no trivial object $(X,1_X)$ with $X\ne\emptyset$ is projective in $\Cal M$.

  \subsection{The pretorsion theory in \ref{8.2} extended to infinite sets.}
Now we generalize the example in \ref{8.2} from the case of sets $X=\{1,2,3,\dots,n\}$ to the case of $X$ any set, possibly infinite.  Let $\Cal M'$  be the category whose objects are all pairs $(X,f)$, where $X$ is any set and $f\colon X\to X$ is any mapping. The morphisms in $\Cal M$ from $(X,f)$ to $(X',f')$ are the mappings $\varphi\colon X\to X'$ such that $f'\varphi=\varphi f$. Thus $\Cal M'$ can be seen as the category (variety) of all algebras $X$ with one operation $f$ that is a unary operation and no axiom. Like in \cite{AL}, it is possible to associate to $f\colon X\to X$ two graphs. The directed graph $G^d_f$ with set of vertices $X$ and one arrow $x\to f(x)$ for each vertex $x\in X$. And the undirected graph $G^u_f$ with set of vertices $X$ and set of edges $L:=\{\,\{x,f(x)\}\mid x\in X,\ x\ne f(x)\,\}$. The undirected graph $G^u_f$ decomposes as a disjoint union of its connected components. 

Let us determine the connected component of the graph $G^u_f$ containing  a fixed vertex $x_0\in X$. The connected component of $x_0$ in $G^u_f$ consists of all vertices $x\in X$ for which there exists a path of finite length from $x_0$ to $x$ in $G^u_f$. In our graph $G^u_f$,  the vertices at a distance $\le 1$ from $x_0$ are exactly those in the set $\{x_0,f(x_0)\}\cup f^{-1}(x_0)$. Hence the connected component of $x_0$ is the closure of $\{x_0\}$ with respect to taking images and inverse images via $f$.
Starting from the fixed vertex $x_0\in X$, we can define a sequence of vertices $x_0,x_1,x_3,\dots$ in $X$ with $x_{i+1}=f(x_i)$ for every $i\ge0$. We have two cases, according to whether the vertices $x_0,x_1,x_3,\dots$ in $X$ are all distinct or not. 

\smallskip

{\em First case: The vertices $x_0,x_1,x_3,\dots$ are all distinct.} In this case, set $B_0:=\{\, x_i\mid i\ge0\,\}$.

\smallskip

{\em Second case: The vertices $x_0,x_1,x_3,\dots$ are not all distinct.} In this case, let $i_0\ge 0$ be the smallest index $i\ge 0$ such that $x_i=x_j$ for some $j>i$. Now let $j_0$ be the smallest index $j>i_0$ with $x_{i_0}=x_{j_0}$. Set $B_0:=\{x_{i_0},x_{i_0+1},\dots, x_{j_0-1}\}$, so that $B_0$ has $j_0-i_0$ elements.

\smallskip

In both cases, we have that $f(B_0)\subseteq B_0$. 
Now recursively define subsets $B_{i+1}$ of $X$ setting $B_{i+1}:= f^{-1}(B_i)$ for every $i\ge 0$. One easily sees by induction that $B_{i+1}\supseteq B_i$ for every $i\ge 0$. 
Thus the chain of subsets $B_i$ of $X$ is ascending. Clearly, the connected component of $G^u_f$ containing $x_0$ is $\bigcup_{i\ge 0}B_i$.

In order to visualize the structure of this connected component of the graph $G^d_f$, notice that $f(B_i)\subseteq B_{i-1}$ for every $i\ge 1$.  Set $A_0:=B_0$ and $A_{i+1}=B_{i+1}\setminus B_i$ for every $i\ge 0$. Then we have a union $\bigcup_{i\ge0}A_i$ of pairwise disjoint sets, $f(A_{i+1})\subseteq A_i$ for every $i\ge0$ and $f(A_0)\subseteq A_0$. We have, like above, two cases. In the first case, in which all the vertices $x_0,x_1,x_3,\dots$ are distinct, $G^d_f$ is a tree. In the second case, in which $A_0=B_0=\{x_{i_0},x_{i_0+1},\dots, x_{j_0-1}\}$ has $j_0-i_0$ elements, the vertices on $A_0$ are on a directed cycle of length $j_0-i_0>0$, and the connected component of $x_0$ in the graph $G^d_f$ is a forest on a cycle. Cf.~\cite{AL}. 

Like in Example \ref{8.2}, there is a functor $U'\colon \Cal M'\to\pre$. It associates to any object $(X,f)$ of $\Cal M'$ the preordered set $(X,\rho_f)$, where $\rho_f$ is defined, for every $x,y\in X$, setting $x\rho_f y$ if $x=f^t(y)$ for some integer $t\ge0$. Any morphism $\varphi\colon (X,f)\to (X',f')$ in $\Cal M'$ is a morphism $\varphi\colon (X,\rho_f)\to (X',\rho_{f'})$ in $\pre$. The functor $U'\colon \Cal M'\to\pre$ is a canonical embedding. Like in Example \ref{8.2}, the functor $U'$ allows us to view  $\Cal M'$ as a subcategory of $\pre$.

We don't go too much into details now, because the proofs are all very similar to those in \cite{AL}. We get that the class $\Cal C':=\Cal M'\cap\Equiv$ consists of all objects $(X,f)$ of $\Cal M'$ for which the relation $\rho_f$ is symmetric, that is, consists of all objects $(X,f)$ for which the connected components of $G^u_f$ are either isolated points or finite circuits. Equivalently, $\Cal C'$ consists of all objects $(X,f)$ of $\Cal M'$ for which $f$ is a bijection and for every $x\in X$ there exists $t>0$ such that $x=f^t(x)$.

The class $\Cal F':=\Cal M'\cap\ParOrd$ consists of all objects $(X,f)$ of $\Cal M'$ for which the relation $\rho_f$ is antisymmetric, that is, consists of all objects $(X,f)$ for which $G^u_f$ is a forest, i.e., $G^u_f$ does not contains circuits (of length $>0$). Equivalently, $\Cal C'$ consists of all objects $(X,f)$ of $\Cal M'$ such that, for every $x\in X$ and and every integer $t>0$, if $x=f^t(x)$, then $x=f(x)$.

The intersection $\Cal Z':=\Cal C'\cap\Cal F'$ consists of all pairs $(X,f)$ with $f\colon X\to X$ the identity mapping of $X$. Hence $\Cal Z'$ is a full subcategory of $\Cal M'$ isomorphic to the category $\Sets$ of sets. A morphism $\varphi\colon (X,f)\to (X',f')$ in $\Cal M'$ is $\Cal Z'$-trivial if and only if it is constant on the circuits of $G^u_f$ and its image $f(X)$ consists of points of $X'$ fixed by $f'$. 

The pair $(\Cal C',\Cal F')$ is a pretorsion theory for $\Cal M'$. For every object $(X,f)$ of $\Cal M'$, the $\Cal Z'$-preexact short exact sequence of Axiom (2) is the sequence $(A_0, f_0)\to (X,f)\to (X/\!\!\sim,\overline{f})$, where $A_0:=\{\,x\in X\mid x=f^t(x)$ for some integer $t\ge1\,\}$, $f_0\colon A_0\to A_0$ is the restriction of $f$, and $\sim$ is the equivalence relation on $X$ (congruence of the universal algebra $(X,f)$), defined, for every $x,y\in X$, by $x\sim y$ if there exist positive integers $t,t'$ such that $x=f^t(y)$ and $y=f^{t'}(x)$.

 \bigskip
  
  \subsection{Finite linearly ordered sets}\label{good}
  
  Every preordered set $(P,\rho)$ is a category with $P$ as its class (set) of objects, for every $p,q\in P$ no morphism $p\to q$ if $p{\!}\not\!\!\rho\,q$, and one morphism $p\to q$ if $p\,\rho\,q$. We will now completely describe all torsion theories in the case of the partially ordered set $X:=\{1,2,3,\dots,n\}$, where $n$ is a fixed positive integer and the preorder is the usual partial order $\le$ on $X$. For $p,q\in X$ with $p\le q$, the unique morphism $p\to q$ will be denoted by $p\le q$ as well. For $p\le q$ in $X$, the subset $\{p,p+1,\dots,q-1,q\}$ of $X$ consisting of all $q-p+1$ integers $t$ with $p\le t\le q$ will be denoted by $[p,q]$ (the {\em interval} from $p$ to $q$). If $A$ is a subset of $X$, we will denote by $A+1$ the set $A+1:=\{\,a+1\mid a\in A,\ a\ne n\,\}$.
  
  \begin{proposition} Let $n$ be a fixed positive integer, $X$ be the linearly ordered set $\{1,2,3,\dots,n\}$ and $T,F$ be subsets of $X$. The following conditions are equivalent.
  
  {\rm (a)} $(T,F)$ is a pretorsion theory for the category $X$.
  
   {\rm (b)} $T\cup F=X$, $1\in T$ and $n\in F$. Moreover, for every $i=1,2,\dots,n-1$, if $i\in T$ and $i+1\in F$, then either $i\in F$ or $i+1\in T$.
   
    {\rm (c)}  There exist two subsets $A$ and $C$ of $X$ such that  $C\cap(A\cup(A+1))=\emptyset$ and $C\cup A\subset X$. Moreover, defining $B:=X\setminus (A\cup C)$, we have that $T=A\cup B$, $F=B\cup C$, $1\in T$ and $n\in F$.\end{proposition}
  
  \begin{proof} (a)${}\Rightarrow{}$(b). Let $(T,F)$ be a pretorsion theory for the category $X$. The $Z$-preexact sequence relative to the object $1$ of $X$ satisfying Axiom (2) must be of the form $t\le 1\le f$. Since $1$ is the least element of $X$, we must have $t=1$. Hence $1\in T$. Similarly, $n\in F$. 
  
  Clearly, $T\cup F\subseteq X$. Conversely, suppose $p\in X$. Set $t:=\max(T\cap[1,p])$ and $f:=\min(F\cap[p,n])$. Notice that such two elements $t$ and $f$ exist because $1\in T$ and $n\in F$. Then $t\le p\le f$. Since every morphism from $t\in T$ to $f\in F$ is $Z$-trivial, there exists $z\in Z:=T\cap F$ with $t\le z\le f$. Hence we have the two cases $t\le z\le p\le f$ and $t\le p\le z\le f$. Suppose $t\le z\le p\le f$. By the maximality of $t$, it follows that $t=z$. From Lemma~\ref{xxx}(a), we get that $p=f$. Hence $p\in F$ in this case. Similarly, if $t\le p\le z\le f$, we have $p\in T$. Therefore $T\cup F=X$. 
  
  Finally, let $i=1,2,\dots,n-1$ be such that $i\in T$ and $i+1\in F$. Since every morphism from $i\in T$ to $i+1\in F$ is $Z$-trivial, there exists $z\in Z$ with $i\le z\le i+1$. Therefore either $i=z$ or $z=i+1$, that is, either $i\in F$ or $i+1\in T$.
  
  (b)${}\Rightarrow{}$(c). Suppose that (b) holds, and set $A:=T\setminus F$, $B:=T\cap F$ and $C:=F\setminus T$. The condition ``for every $i=1,2,\dots,n-1$, if $i\in T$ and $i+1\in F$, then either $i\in F$ or $i+1\in T$'' can be restated as `` there does not exist $i=1,2,\dots,n-1$ such that $i\in T$, $i+1\in F$, $i\notin F$ and $i+1\notin T$'', i.e., ``$((T\setminus F)+1)\cap (F\setminus T)=\emptyset$'', that is, ``$(A+1)\cap C=\emptyset$''. Hence, from (b), we have that the two subsets $A:=T\setminus F$and $C:=F\setminus T$ of $X$ are such that  $C\cap(A\cup(A+1))=\emptyset$ and $C\cup A\subset X$, because $X\setminus(C\cup A)=Z$. Moreover, we have that $B=T\cap F=Z$, $B=X\setminus (A\cup C)$, $T=A\cup B$ and $F=B\cup C$.
  
   (c)${}\Rightarrow{}$(b). Suppose that (c) holds, so  that we have subsets $A,B,C,T,F$ of $X$ such that $C\cap(A\cup(A+1))=\emptyset$, $B=X\setminus (A\cup C)\ne\emptyset$, $T=A\cup B$ and $F=B\cup C$. Then $T\cup F=A\cup B\cup C=X$.  Finally, as we have seen in the proof of     (b)${}\Rightarrow{}$(c), the condition ``for every $i=1,2,\dots,n-1$, if $i\in T$ and $i+1\in F$, then either $i\in F$ or $i+1\in T$'' is equivalent to $(A+1)\cap C=\emptyset$, which holds by (c).
  
   (b) \& (c)${}\Rightarrow{}$(a). If (b) and (c) hold, then $Z:=T\cap F=B\ne\emptyset$. In order to prove that every morphism from $t\in T$ to $f\in F$ is $Z$-trival, we must show that  there exists $z\in Z$ with $t\le z\le f$. Suppose the contrary. Then $[t,f]\subseteq A\cup C$. Let $\overline{t} $ be the greatest element in $[t,f]\cap A$. Then $\overline{t} +1\in C$. This contradicts the condition  $C\cap(A\cup(A+1))=\emptyset$. Hence Axiom (1) in the definition of  pretorsion theory is satisfied. As far as Axiom (2) is concerned, fix an element $p\in X$. Since $X$ is the disjoint union of $A$, $B$ and $C$, we have three possible cases:
   
   If $p\in A$, then the $Z$-preexact sequence relative to $p$ satisfying Axiom (2)  is $p\le p\le z$, where $z:=\min(B\cap[p,n])$. Such an element $z$ exists because $B\cap[p,n]\ne\emptyset$. Otherwise $[p,n]\subseteq A\cup C$. Let $\overline{p}$ be the greatest element in $[p,n]\cap A$. Then $\overline{p} +1\in C$. (Notice that $n\in C$.) This contradicts the condition  $C\cap(A\cup(A+1))=\emptyset$. 
   
   If $p\in B$, then the $Z$-preexact sequence relative to $p$ satisfying Axiom (2)  is $p\le p\le p$, trivially.
   
   If $p\in C$, then the $Z$-preexact sequence relative to $p$ satisfying Axiom (2)  is $z\le p\le p$, where $z:=\max(B\cap[1,p])$. Such an element $z$ exists because $B\cap[1,p]\ne\emptyset$. Otherwise $[1,p]\subseteq A\cup C$. Let $\overline{p}$ be the greatest element in $[1,p]\cap A$.  (Notice that now $1\in A$ and $p\in C$.) Then $\overline{p} +1\in C$. This contradicts the condition  $C\cap(A\cup(A+1))=\emptyset$.\end{proof}

\subsection{Topological groups}\label{Tg}

Let $\Cal C$ be the category $\mathsf{Grp(Top)}$ of topological groups, $\Cal T$ the category $\mathsf{Grp(Ind)}$ of topological groups endowed with the trivial topology, and $\Cal F$ the category of $\mathsf{Grp(T_0)}$ of $\mathsf{T_0}$ topological groups. Then $(\mathsf{Grp(Ind)}, \mathsf{Grp(T_0)})$ is a pretorsion theory in $\mathsf{Grp(Top)}$ \cite{BG}. In this example, $\Cal Z$ consists of topological groups with one element, and $\Cal Z$-trivial morphisms are the morphisms $\alpha \colon (G,\cdot, \tau_G) \to (H, \cdot, \tau_H)$ such that $\alpha(g) = 1_H$ for every $g\in G$. The canonical short (pre)exact sequence relative to a topological group $(G,\cdot, \tau_G)$ is 
$$
\xymatrix{0 \ar[r] & (\overline{\{1_G \}},\cdot, \tau_{i})  \ar[r]^{j} & (G,\cdot, \tau_G) \ar[r]^-{\pi} & ({G}/{ \overline{\{1_G \}}}, \cdot, \tau_q) \ar[r] & 0
}
$$
where $(\overline{\{1_G \}},\cdot, \tau_{i})$  is the closure of the trivial subgroup $\{1_G\}$ in $(G,\cdot, \tau_G)$ (with the induced topology), and $\pi$ is the canonical mapping.

Another example of pretorsion theory in the category $\mathsf{Grp(Top)}$ of topological groups is given by the pair $(\mathsf{Grp(Conn)}, \mathsf{Grp(TotDis)})$, where $\mathsf{Grp(Conn)}$ is the category of connected groups and $\mathsf{Grp(TotDis)}$ is the category of totally disconnected groups.
In this case the canonical short exact sequence corresponding to a topological group $(G,\cdot, \tau_G)$ is 
$$
\xymatrix{0 \ar[r] & (\Gamma_{1_G}, \cdot, \tau_i)  \ar[r]^{l} & (G,\cdot, \tau_G) \ar[r]^{\varphi} & ({G}/{ \Gamma_{1_G}}, \cdot, \tau_q) \ar[r] & 0
}
$$
where $(\Gamma_{1_G}, \cdot, \tau_i)$ is the connected component of the unit $1_G$ of the topological group $(G,\cdot, \tau_G)$, and $\varphi$ is the canonical mapping.

\subsection{The Kolmogorov quotient of a topological space}

This is a generalization of the example in the first paragraph of \ref{Tg}. Let $\Cal C=\Top$ be the category of topological spaces and $\mathsf{T_0}$ be the subcategory of all $T_0$ topological spaces. A topology on a set $X$ is a {\em partition topology} if there is an equivalence relation on $X$ for which the equivalence classes are a base of open sets for the topology. Let $\mathsf{P}$ be the full subcategory of $\Top$ whose objects are all topological spaces whose topology is a partition topology. Clearly, $\Cal Z:=\mathsf{P}\cap\mathsf{T_0}$ consists of all discrete topological spaces. We claim that $(\mathsf{P},\mathsf{T_0})$ is a pretorsion theory in $\Top$. It is easily seen that: (1) If $X,Y$ are topological spaces, a continuous function $f\colon X\to Y$ is $\Cal Z$-trivial if and only if there is a partition of $X$ into clopen subsets such that $f$ is constant on each of these clopen subsets; (2)
If $X$ is a topological space whose topology is a partition topology, $Y$ is a $T_0$ space and $f\colon X\to Y$ is a continuous mapping, then $f$ is $\Cal Z$-trivial. 

Let $(X,\tau)$ be any topological space. If $x,y$ are two points of $X$, we write $x\equiv y$ if, for each open set $U$, $U$ contains either both $x$ and $y$ or neither of them (that is, if $\overline{\{x\}}=\overline{\{y\}}$). This relation $\equiv$ on $X$ is an equivalence relation, called {\em topological indistinguishability.}

Let  $KX := X/\!\!\equiv$ denote the quotient space, endowed with the quotient topology. The topological space $KX$ is the {\em Kolmogorov quotient} of $X$. The space $KX$ is $T_0$, the quotient map $q\colon X\to KX$ is open and induces a bijection between the topology on $X$ and the topology on $KX$. Let $\pi_\tau$ be the partition topology on $X$ induced by the equivalence relation $\equiv$. The identity mapping $k\colon (X,\pi_\tau)\to(X,\tau)$ is continuous. Then, in order to prove that $(\mathsf{P},\mathsf{T_0})$ is a pretorsion theory in $\Top$, it will be enough to verify that the sequence of morphisms $$\xymatrix{
	(X,\pi_\tau) \ar[r]^k &  (X,\tau) \ar[r]^q &  KX}$$ in $\Top$ is $\Cal Z$-preexact. 
Since $\pi_\tau$ is a partition topology and $KX$ is a $T_0$ space, it follows that $qk$ is $\Cal Z$-trivial.

 Let $\lambda\colon (X,\tau)\to Y$ be any morphism in $\Top $ such that $\lambda k=\lambda\colon (X,\pi_\tau)\to Y$ is $\Cal Z$-trivial. Let $u\colon (X,\pi_\tau)\to Z,\ v\colon Z\to Y$ be morphisms such that $\lambda k=vu$, where $Z$ is some discrete space. If $x,y\in X$ are such that $x\equiv y$, consider the open neighborhood $u^{-1}(\{u(x)\})$ of $x$ (here we are using the facts that $u$ is continuous and that $Z$ is discrete). Since $\pi_\tau$ is the partition topology induced by $\equiv$, it follows $[x]_\equiv=[y]_\equiv\subseteq u^{-1}(\{u(x)\})$. In particular, $y\in u^{-1}(\{u(x)\})$ and thus, since $\lambda=vy$, $\lambda(x)=\lambda(y)$. This proves that $\lambda$ induces the mapping $\lambda_0\colon KX\to Y$, $[x]_\equiv\mapsto \lambda(x)$, which is the unique mapping such that $\lambda=\lambda_0q$. Finally, $\lambda_0$ is continuous, essentially because $q$ is surjective and open and $\lambda=\lambda_0 q$. This proves that $q$ is a $\cal Z$-precokernel of $k$. 
 Similarly, one easily proves that $k$ is a $\cal Z$-prekernel of $q$. 
 
 \medskip
 
This example of pretorsion theory $(\mathsf{P},\mathsf{T_0})$ in $\Top$ extends both the first example in ~\ref{Tg} and the example of pretorsion theory  $(\Equiv,\ParOrd)$ on $\pre$  (Alexandrov spaces). By ``extension'' we mean via the forgetful functor $\mathsf{Grp(Top)}\to\Top$ and the embedding functor of the subcategory of Alexandrov spaces into $\Top$.

\subsection{Totally disconnected topological spaces} Let $\Top$ be the category of all topological spaces. Recall that the {\em connected components} of a topological space $X$ are the maximal connected subsets of $X$. They are closed subsets of $X$ and form a partition of $X$. Let $\Cal F$ be the category of all totally disconnected topological spaces, that is, those for which the connected components are the one-point sets. Let $\Cal T$ be the category of all topological spaces for which the connected components are open subsets. It is possible to check that $(\Cal T,\Cal F)$ is a pretorsion theory in $\Top$. Now $\Cal Z:=\Cal T\cap\Cal F$ is the category of all discrete topological spaces. For any topological space $X$, let $\sim$ be the equivalence relation on $X$ whose equivalence classes are the connected components of $X$. Endow the quotient set $X/\!\!\sim$ with the quotient topology, i.e.,  the finest topology making the canonical projection $q\colon X\to X/\!\!\sim$ continuous. Then $f(X):=X/\!\!\sim$ is totally disconnected. For any continuous map $\varphi\colon X\to Y$ to a totally disconnected space $Y $, there exists a unique continuous map $\widetilde{f}\colon X/\!\!\sim{}\to{}\, Y$ with $\varphi=\widetilde{\varphi}q$. Notice that if $T\in\Cal T$, then $T/\!\!\sim$ is a discrete topological space, so that every continuous mapping $\varphi\colon T\to Y$ with $Y\in\Cal F$ factors through the discrete topological space $T/\!\!\sim$, hence is $\Cal Z$-trivial. For any topological space $(X,\tau)$, let $\sigma_\tau$ be the coarsest topology weaker than both $\tau$ and the partition topology on $X$ determined by the partition into connected components. A base for the topology $\sigma_\tau$ is given by the sets conn$(x)\cap U$, where conn$(x)$ denotes the connected component of any element $x\in X$ and $U$ is any open subset in the topology $\tau$. The short $\Cal Z$-preexact sequence corresponding to any topological space $(X,\tau)$ is the sequence $\xymatrix{
    (X,\sigma_\tau) \ar[r]^i &  (X,\tau) \ar[r]^q & X/\!\!\sim}$, where $i\colon (X,\sigma_\tau) \to (X,\tau) $ is the identity.

\section{Final remarks. Other torsion theories in the literature.}
  
In this last section, we briefly compare our approach with the  torsion theories studied in \cite{BG, CMM, GJ,GJM, JT}.

\begin{remark}\label{homological-cat}
	{\rm Let $\cal H$ be a \emph{homological category}, i.e., a category that is pointed (with a zero object $0$), regular and protomodular. Following \cite[Definition 4.1]{BG}, a torsion theory in $\cal H$ is a pair $(\cal T,\cal F)$ of full replete subcategories of $\cal H$ satisfying the following axioms. 
		\begin{enumerate}
			\item For every $T\in\cal T,F\in\cal F$, $\Hom_{\cal H}(T,F)=\{0\}$ { (here $0$ denotes the zero morphism $T \rightarrow 0 \rightarrow F$ from $T$ to $F$)}.
			\item For every object $X\in \cal C$, there exists a short exact sequence $$0\to T\to X\to F\to 0,$$
			where $T\in\cal T$ and $F\in\cal F$.
		\end{enumerate}
		In view of \cite[Lemma 4.3(1)]{BG}, if $(\cal T,\cal F)$ is a torsion theory for $\cal H$, then $\cal Z$ {  consists of (the replete subcategory of) the zero object $0$ of $\cal C$}. {The fact that $\Cal Z = \{ 0 \}$ implies that $\cal Z$-preexact sequences in $\cal H$ are classical exact sequences, so that every torsion theory in a homological category is a pretorsion theory in the sense of this article.}
	}

{\rm 
{ Of course, in order to define a torsion theory as here above, it suffices to assume that $\cal H$ is a \emph{pointed 
category}, as was done by Janelidze and Tholen in \cite{JT}. Also this general notion of torsion theory
is a special case of our notion of pretorsion theory,}{ essentially for the same reason as that recalled above:
when $\cal Z = \{ 0 \}$, $\cal Z$-preexact sequences reduce to the usual exact sequences. Note that the following result, proved in \cite{JT},
is a relevant special case of Proposition \ref{idempotent-subfunctor}:}
\begin{corollary}[\cite{JT}] 
For a full replete subcategory $\cal F$ of a {pointed category $\cal C$ with kernels and cokernels} the following conditions are equivalent. 
\begin{enumerate}
	\item $\cal F$ is the torsion-free part of a torsion theory $(\cal T,\cal F)$ in $\cal C$. 
	\item $\cal F$ is a normal epireflective subcategory of $\cal C$, and the radical induced by the reflector $f \colon \cal C \rightarrow \cal F$ is idempotent.
\end{enumerate}
\end{corollary}
}
\end{remark}
\begin{remark}{\rm Another notion of torsion theory is studied in \cite{GJ}. In that paper, torsion theories in multi-pointed categories are studied. In those categories, the identity of any object has a $\Cal Z$-prekernel and a $\Cal Z$-precokernel. In this paper, we study a case more general than that, because, in our case, it is possible that there are objects $B$ in $\Cal C$ such that $\Hom_{\Cal C}(Z,B)=\emptyset$ for every $Z\in\Cal Z$ (Remark~\ref{non-empty}(b)). For such an object $B$, the identity $1_B\colon B\to B$ cannot have a $\Cal Z$-prekernel.

Moreover, in the torsion theories studied in \cite{GJ}, the category $\Cal Z$ turns out to be a reflective and coreflective subcategory of $\Cal C$, and this is not true in our case. 
For instance, if, like in the previous paragraph, there is an object $B$ in $\Cal C$ such that $\Hom_{\Cal C}(Z,B)=\emptyset$ for every $Z\in\Cal Z$, the embedding $\Cal Z\hookrightarrow\cal C$ cannot have a right adjoint.
Let us also mention the fact that this approach to torsion theories using ideal of morphisms was already adopted in \cite{Mantovani}. In the article \cite{G3}, torsion theories in the context of \emph{semiexact categories} were studied.}
\end{remark}

\begin{remark}{\rm In Monoid Theory, there are three apparently related notions:

\noindent (1) {\em Congruences on a monoid $M$,} that is, equivalence relations on the set $M$ compatible with the binary operation $\cdot$ and the nullary operation $1$ on $M$. { All such congruences arise as \emph{kernel pairs} of some monoid morphism $f\colon M\to N$, i.e., as equivalence relations on $M$ of the form $\sim_f$, where $x\sim_fy$ if $f(x)=f(y)$.

\noindent (2) {\em The equivalence class $[1_M]_\sim$ of the identity $1_M$ of $M$ with respect to a congruence $\sim$ on $M$,} {  i.e., the subset $f^{-1}(1_N)$
for some monoid morphism $f\colon M\to N$.}

\noindent (3) {\em Ideals of $M$,} that is, non-empty subsets $I$ of $M$ such that, for every $i\in I$ and every $x\in M$, both $xi$ and $ix$ belong to $I$.

Clearly, for a monoid $M$ congruences as in (1) determine equivalence classes as in (2), but not conversely. If the monoid $M$ is a group, the converse holds, as is well known, and the equivalence classes as in (2) are the normal subgroups. { The same property holds, more generally, in any \emph{ideal determined category} \cite{ID}.}

On the other hand, ideals as in (3) are related to the notion of monoids with zero.  For every monoid morphism $f\colon M\to N$, where $N$ is a monoid with zero $0_N$, $f^{-1}(0_N)$ is either the empty set or an ideal of $M$. Given any ideal $I$ of $M$, it is possible to define a congruence on $M$ setting, for every $x,y\in M$, $x\sim_Iy$ if either $x=y$, or both $x\in I$ and $y\in I$. Moreover, for any ideal $I$ of $M$, the quotient monoid $M/\!\!\sim_I$ turns out to be a monoid with zero. For a group, the unique ideal is the improper subset.

For rings $R$ with identity $1_R$, ring morphisms that map $1$ to $1$, and congruences in (1) compatible with the ring operations, we have that there is a one-to-one correspondence between congruences (compatible with $+$ and $\cdot$) and subsets as in (2), which are exactly the ideals $I$ of the ring as in (3). One has that $f^{-1}(1_N)=1_M+f^{-1}(0_N)$. Thus (1), (2) and (3) are equivalent concepts for rings.

All this has its counterpart in categories, where we have the notion of congruence $R$ on a category $\Cal C$ \cite[p.~52]{ML}, and for any congruence $R$ it is possible to construct the quotient category $\Cal C/R$ \cite[Section II.8]{ML}. Here a congruence $R$ is a function that assigns to each pair of objects $a,b$ of $\Cal C$ an equivalence relation $R_{a,b}$ on the set $\Hom_{\Cal C}(a,b)$, in a way compatible with composition. An ideal $\Cal I$ in a category $\Cal C$ is a function that assigns to each pair of objects $a,b$ of $\Cal C$ a subset $\Cal I_{a,b}$ of the set $\Hom_{\Cal C}(a,b)$ so that, for all objects $a,b,c,d $ of $\Cal C$, if $f\in \Hom_{\Cal C}(a,b)$, $ g\in{\Cal I}_{b,c}$, $h\in\Hom_{\Cal C}(c,d)$, then $hgf\in {\Cal I}_{a,d}$. For every ideal $\Cal I$ in a category $\Cal C$, there is an associated congruence $R$ on $\Cal C$ defined, for every $f,g\in \Hom_{\Cal C}(a,b)$, by $fR_{a,b}g$ if either $f=g$ or both $f$ and $g$ belong to $\Cal I_{a,b}$. Notice that ideals of a category $\Cal C$ are exactly  the subfunctors of the functor $\Hom\colon \Cal C^{\op}\times\Cal C\to\Sets$, and quotient categories $\Cal C/R$, with $ R$ a congruence in $\Cal C$, are exactly  the quotient functors of the functor $\Hom\colon \Cal C^{\op}\times\Cal C\to\Sets$.}

We can define a notion of {\em quasi-zero object} in a category $\Cal C$, in the same spirit as \emph{quasi-terminal objects} were defined in \cite[A1.5.14]{Jo}. They are the objects $z$ of $\Cal C$ such that there is at most one morphism from, and to, any object $a$ in $\cal C$. In other words, 
$|\Hom_{\Cal C}(a,z)|\le 1$ and $|\Hom_{\Cal C}(z,a)|\le 1$ for every object $a$ of $\Cal C$. Suppose we are in the special case studied in this paper, in which $\Cal C$ is a category,  $\Cal Z$ is a non-empty class of objects of $\Cal C$, and $\Cal I$ is the ideal of $\Cal C$ generated by the identity morphisms of the objects in $\Cal Z$. If $R$ is the congruence associated to $\Cal I$, then all objects of $\Cal Z$ become quasi-zero objects in the quotient category $\Cal C/R$.

Finally, when the category $\Cal C$ is additive, there is a one-to-one correspondence between the congruences on $\Cal C$ that are compatible not only with composition but also with addition, and the ideals $\Cal I$ in $\Cal C$ for which $\Cal I_{a,b}$ is a subgroup of the group $\operatorname{Hom}_{\Cal C}(a,b)$ for every pair $a,b$ of objects of $\Cal C$.}\end{remark}


\begin{thebibliography}{99}

\bibitem{B} M. Barr, \emph{Non-abelian torsion theories,} Canad. J. Math. \textbf{25} (1973), 1224--1237, https://doi.org/10.4153/CJM-1973-130-4. 

\bibitem{BG} D. Bourn and M. Gran, \emph{Torsion theories in homological categories}, J. Algebra \textbf{305} (2006),  18--47, https://doi.org/10.1016/j.jalgebra.2006.07.011. 

\bibitem {BGV} A. Buys, N. J. Groenewald and S. Veldsman, {\em Radical and semisimple classes in categories,} Quaestiones Math. \textbf{4}(3) (1980--81), 205--220, https://doi.org/10.1080/16073606.1981.9631873.

\bibitem {BV} A. Buys and S. Veldsman, {\em Quasiradicals and radicals in categories,} Publ. Inst. Math. (Beograd) (N.S.) \textbf{38}(52) (1985), 51--63. 

\bibitem {CDT} M. M. Clementino, D. Dikranjan and W. Tholen, {\em Torsion theories and radicals in normal categories,} J. Algebra \textbf{305} (2006), 92--129,  https://doi.org/10.1016/j.jalgebra.2005.09.030.

\bibitem{CMM} M. M. Clementino, N. Martins-Ferreira and A. Montoli, {\em On the categorical behaviour of preordered groups,} J. Pure Appl. Algebra 
\textbf{223}(10) (2019), 4226--4245,  https://doi.org/10.1016/j.jpaa.2019.01.006.

\bibitem{Dickson} S. E. Dickson, {\em A torsion theory for abelian categories,} Trans. Amer. Math. Soc. \textbf{121} (1966), 223--235,  https://www.jstor.org/stable/1994341.

\bibitem{Eh} C. Ehresmann, {\em Cohomologie a valeurs dans une cat\'egorie domin\'ee,} Extraits du Colloque de Topologie, Bruxelles 1964, in: C. Ehresmann, ``Oeuvres compl\`etes et comment\'ees'', Partie III-2, Amiens, 1980, pp. 531--590.

\bibitem{EG} T. Everaert and M. Gran, {\em Protoadditive functors, derived torsion theories and homology,} J. Pure Appl. Algebra \textbf{219}(8) (2015), 3629--3676, https://doi.org/10.1016/j.jpaa.2014.12.015.
 
\bibitem{AC} A. Facchini and C. Finocchiaro, {\em Pretorsion theories, stable category and preordered sets,} Ann. Mat. Pura Appl. \textbf{199} (2020), 1073--1089, https://doi.org/10.1007/s10231-019-00912-2.

\bibitem{FFG} A. Facchini, C. Finocchiaro and M. Gran, {\em A new Galois structure in the category of internal preorders}, Theory Appl. Categories \textbf{35}(11) (2020), 326--349.

\bibitem{AL} A. Facchini and L. Heidari Zadeh, {\em An extension of properties of symmetric group to monoids and a pretorsion theory in the category of mappings,} J. Algebra Appl. \textbf{18}(12) (2019), 1950234 https://doi.org/10.1142/S0219498819502347.

\bibitem{GJU} M. Gran, Z. Janelidze and A. Ursini, {\em A good theory of ideals in regular multi-pointed categories}, J. Pure Appl. Algebra  \textbf{216}(8-9) (2012), 1905--1919,  https://doi.org/10.1016/j.jpaa.2012.02.028.

\bibitem{GJR} M. Gran, Z. Janelidze and D. Rodelo, {\em $3 \times 3$-Lemma for star-exact sequances}, Homology, Homotopy Appl. \textbf{14}(2) (2012), 1--22.

\bibitem{G1} M. Grandis, {\em A categorical approach to exactness in algebraic topology,} in: ``Atti del
V Convegno Nazionale di Topologia'', Lecce-Otranto 1990, Rend. Circ. Mat. Palermo \textbf{29} (1992), 179--213.

\bibitem{G2} M. Grandis, {\em On the categorical foundations of homological and homotopical algebra,}
Cah. Topol. G\'eom. Diff. Categ. \textbf{33} (1992), 135--175.

\bibitem{G3} M. Grandis, ``Homological Algebra in strongly non-abelian settings'', World Scientific
Publishing Co., Singapore, 2013,  https://doi.org/10.1142/8608

\bibitem{GJ} M. Grandis and G. Janelidze, {\em From torsion theories to closure operators and  factorization systems,} Categories and General Algebraic Structures with Applications, \textbf{12}(1) (2020), 89--121.

\bibitem{GJM} M. Grandis, G. Janelidze and L. M\'arki, {\em Non-pointed exactness, radicals, closure operators,} J. Aust. Math. Soc. \textbf{94}(3) (2013),  348--361, doi:10.1017/S1446788713000086

\bibitem{ID} G. Janelidze, L. M\'arki, W. Tholen and A. Ursini, {\em Ideal determined categories},
Cah. Top. G\'eom. Diff. Cat\'eg. \textbf{51}(2) (2010), 115--125.
 
\bibitem{JT} G. Janelidze and W. Tholen, {\em Characterization of torsion theories in general categories}, in  ``Categories in algebra, geometry and mathematical physics'', A. Davydov, M. Batanin, M. Johnson, S. Lack and A. Neeman Eds., Contemp. Math. \textbf{431}, Amer. Math. Soc., Providence, RI, 2007, pp. 249--256. 

\bibitem{Jo} P.T. Johnstone, ``Sketches of an elephant: a topos theory compendium'', Vol. 1, Oxford Logic Guides \textbf{43}, Oxford Univ. Press, New York, 2002,  doi:10.1017/S1079898600003462

\bibitem{La} R. Lavendhomme, {\em Un plongement pleinement fid\`ele de la cat\'egorie des groupes,} Bull.
Soc. Math. Belgique \textbf{17} (1965), 153--185.

\bibitem{ML} S. MacLane,  ``Categories for the Working Mathematician'', 2nd edn., Springer-Verlag, New York-Berlin, 1998,  https://doi.org/10.1007/978-1-4612-9839-7.

\bibitem{Mantovani} S. Mantovani, {\em Torsion theories for crossed modules,} invited talk at the ``Workshop on category theory and topology'', Universit\'e catholique de Louvain, September 2015.

\bibitem{BR} B. A. Rattray, \emph{Torsion theories in non-additive categories,} Manuscripta Math. \textbf{12} (1974), 285--305,  https://doi.org/10.1007/BF01155518.

\bibitem{RT} J. Rosick\'y and W. Tholen, \emph{Factorization, fibration and torsion}, J. Homotopy Relat. Struct. \textbf{2} (2007), 295--314, https://eudml.org/doc/230964.

\bibitem{V82} S. Veldsman,  \emph{On the characterization of radical and semisimple classes in categories,} Comm. Algebra \textbf{10}(9) (1982), 913--938, DOI: 10.1080/00927878208822757.

\bibitem {V84} S. Veldsman,  \emph{Radical classes, connectednesses and torsion theories,} Suid-Afrikaanse Tydskr. Natuurwetenskap Tegnol. \textbf{3}(1) (1984), 42--45  https://doi.org/10.4102/satnt.v3i1.1066.

\bibitem {VW} S. Veldsman and R. Wiegandt,  \emph{On the existence and nonexistence of complementary radical and semisimple classes,} Quaestiones Math. \textbf{7}(3) (1984), 213--224  DOI: 10.1080/16073606.1984.9632332.

\end{thebibliography}
\end{document}